\documentclass[a4paper,12pt,twoside]{article}
\date{August 28, 2020}

\usepackage[english]{babel}
\usepackage{lmodern,fancyhdr,lastpage,nccfoots,caption}
\usepackage{hyperref}
\hypersetup{
    pdfstartview={FitH},     									
    pdftitle={VHS of rank 3 k-Higgs bundles},             				
    pdfauthor={Ronald Alberto Z\'u\~niga-Rojas},     						
    pdfsubject={Vector bundles on curves and their moduli},           				
    pdfkeywords={Higgs bundles, Hodge bundles, holomorphic triples, moduli spaces, stable triples, variations of Hodge structure, vector bundles.},
    pdfnewwindow=true,        									
    colorlinks=true,          									
    linkcolor=blue,           									
    citecolor=magenta,            									
    urlcolor=blue             									
}
\usepackage[utf8]{inputenc}
\usepackage[all]{xy}
\usepackage{array}
\usepackage{graphicx,color}
\usepackage{amsmath,amssymb,amsthm}
\usepackage{enumerate}
\usepackage[a4paper,margin=1in]{geometry}
\usepackage{textcomp}
\usepackage{mathrsfs}

\fancypagestyle{plain}{
\fancyhf{}}

\DeclareMathOperator{\End}{End}     
\DeclareMathOperator{\GCD}{GCD}     
\DeclareMathOperator{\Hom}{Hom}     
\DeclareMathOperator{\rk}{rk}       
\DeclareMathOperator{\Sym}{Sym}     
\DeclareMathOperator{\tr}{tr}       
\DeclareMathOperator{\divv}{div}    

\newcommand{\al}{\alpha}           
\newcommand{\la}{\lambda}           
\newcommand{\Om}{\varOmega}         
\newcommand{\Sg}{\Sigma}            
\newcommand{\sg}{\sigma}            

\newcommand{\bC}{\mathbb{C}}        
\newcommand{\bN}{\mathbb{N}}        
\newcommand{\bQ}{\mathbb{Q}}        
\newcommand{\bR}{\mathbb{R}}        
\newcommand{\bS}{\mathbb{S}}        
\newcommand{\bZ}{\mathbb{Z}}        
\newcommand{\bP}{\mathbb{P}}        

\newcommand{\sA}{\mathcal{A}}       
\newcommand{\sE}{\mathcal{E}}       
\newcommand{\gG}{\mathcal{G}}       
\newcommand{\sJ}{\mathcal{J}}	    
\newcommand{\sM}{\mathcal{M}}       
\newcommand{\sN}{\mathcal{N}}       
\newcommand{\sO}{\mathcal{O}}       

\newcommand{\td}{\tilde{d}}	    
\newcommand{\tE}{\tilde{E}}	    

\newcommand{\dd}{\mathbf{d}}        
\newcommand{\rr}{\mathbf{r}}        

\renewcommand{\geq}{\geqslant}        
\newcommand{\hookto}{\hookrightarrow} 
\renewcommand{\leq}{\leqslant}        
\newcommand{\ox}{\otimes}             
\newcommand{\sse}{\subseteq}		
\renewcommand{\:}{\colon}             

\newcommand{\word}[1]{\quad\text{#1}\quad} 

\def\longto^#1{\xrightarrow{\;#1\;}} 

\theoremstyle{plain}
\newtheorem{Th}{Theorem}[section]   
\newtheorem*{nonum-Th}{Theorem}     
\newtheorem{Prop}[Th]{Proposition}  
\newtheorem{Lem}[Th]{Lemma}         
\newtheorem{Cor}[Th]{Corollary}     
\newtheorem*{nonum-Cor}{Corollary}  

\theoremstyle{definition}
\newtheorem{Def}[Th]{Definition}    

\theoremstyle{remark}
\newtheorem{Rmk}[Th]{Remark}        

\numberwithin{equation}{section}

\newcommand*{\QEDA}{\hfill\ensuremath{\boxminus}} 
\DeclareRobustCommand{\QEDA}{\ifmmode
  \else \leavevmode\unskip\penalty9999 \hbox{}\nobreak\hfill \fi
  \quad\hbox{\qedasymbol}}
\newcommand{\qedasymbol}{$\boxminus$} 

\newcommand{\hideqed}{\renewcommand{\qed}{}} 

\makeatletter
\renewcommand{\section}{\@startsection{section}{1}{\z@}%
                        {-3.5ex \@plus -1ex \@minus -.2ex}%
                        {2.3ex \@plus.2ex}%
                        {\normalfont\large\bfseries}}
\renewcommand{\subsection}{\@startsection{subsection}{2}{\z@}%
                        {-3.25ex \@plus -1ex \@minus -.2ex}%
                        {1.5ex \@plus .2ex}%
                        {\normalfont\normalsize\bfseries}}
\renewcommand{\subsubsection}{\@startsection{subsubsection}{3}{\z@}%
                        {-3.25ex \@plus -1ex \@minus -.2ex}%
                        {1.5ex \@plus .2ex}%
                        {\normalfont\normalsize\itshape}}
\renewcommand{\@dotsep}{200} 
\makeatother

\begin{document}
 
\thispagestyle{empty}

\begin{center}
\Large
\textsc{Variations of Hodge Structures\\ of Rank Three $k$-Higgs Bundles\\ and Moduli Spaces of Holomorphic Triples}\\

\bigskip
\normalsize
  August 28, 2020\\
  
\bigskip
\emph{Ronald A. Z\'u\~niga-Rojas}\footnote{ \scriptsize
  Supported by Universidad de Costa Rica through Escuela de Matem\'atica, specifically through CIMM 
  (Centro de Investigaciones Matem\'aticas y Metamatem\'aticas), Project {\tt 820-B8-224}. This work is partly based on the Ph.D. Project~\cite{z-r} 
  called {\em ``Homotopy Groups of the Moduli Space of Higgs Bundles''}, supported by FEDER through Programa Operacional Factores de 
  Competitividade-COMPETE, and also supported by FCT (Funda\c{c}\~ao para a Ci\^encia e a Tecnologia) through the projects {\tt PTDC/MAT-GEO/0675/2012} 
  and {\tt PEst-C/MAT/UI0144/2013} with grant reference {\tt SFRH/BD/51174/2010}.}\\[6pt] 
\small Centro de Investigaciones Matem\'aticas y Metamatem\'aticas CIMM\\
\small Escuela de Matem\'atica, Universidad de Costa Rica,\\ 
\small San Jos\'e 11501, Costa Rica\\
\small e-mail: \texttt{ronald.zunigarojas@ucr.ac.cr}\\
\small ORCID: \href{https://orcid.org/0000-0003-3402-2526}{\texttt{0000-0003-3402-2526}}
\end{center}

\begin{abstract}

\noindent There is an isomorphism between the moduli spaces of $\sg$-stable holomorphic triples and some of the critical submanifolds of the moduli space of 
$k$-Higgs bundles of rank three, whose elements $(E,\varphi^k)$ correspond to variations of Hodge structure, VHS.
There are special embeddings on the moduli spaces of $k$-Higgs bundles of rank three. The main objective here is to study the cohomology of the 
critical submanifolds of such moduli spaces, extending those embeddings to moduli spaces of holomorphic triples.

\end{abstract}

\begin{flushleft}

\emph{MSC classes}:  Primary \texttt{14F45}; Secondaries \texttt{14D07}, \texttt{14H60}.
\end{flushleft}

\section*{Introduction} 
\addcontentsline{toc}{section}{Introduction}
\label{sec:0} 

Consider a compact connected Riemann surface $X$ of genus $g \geq 2$. Algebraically, $X$ is a complete irreducible non-singular curve over $\bC$. 
Let $\sN = \sN(r,d)$ be the moduli space of polystable vector bundles of rank $r$ and degree $d$ over $X$. In this paper, we consider the co-prime 
condition $\GCD(r,d) = 1$, which ensures that polystable implies stable. This space has been widely worked by Atiyah~\&~Bott~\cite{atbo}, 
Desale~\&~Ramanan~\cite{dera}, Earl~\&~Kirwan~\cite{eaki}, among other authors. Here, we consider it as the corresponding minimal critical submanifold 
of $\sM(r,d)$, the moduli space of polystable Higgs bundles. A Higgs bundle over $X$ is a pair $(E,\varphi)$ where $E\to X$ is a holomorphic vector 
bundle and $\varphi\: E \to E\otimes K$ is an endomorphism twisted by the cotangent bundle $K = T^{*}X$. Fixing rank $r$ and degree $d$ of the 
underlying vector bundle $E$, the isomorphism classes of polystable Higgs bundles are parametrized by a quasiprojective variety: the moduli space of 
polystable Higgs bundles $\sM^{ps}(r,d)$. Again, since $\GCD(r,d) = 1$, polystability implies stability and then, the space 
$\sM^{ps}(r,d) = \sM^{ss}(r,d) = \sM^{s}(r,d)$ becomes a smooth projective variety. These spaces were first worked by Hitchin~\cite{hit2} and 
Simpson~\cite{sim2}. Since then, they have been around for more than thirty years, and have been studied extensively by a number of authors: \emph{e.g.}
\cite{decataldo-mark-hausel-migliorini:2012,
hausel:2013,
hausel-letellier-rodriguez-villegas:2011,
hausel-rodriguez-villegas:2008,
hath1,hath2,
z-r0}.

\noindent Higgs bundles are an interesting topic of research because they have links with many other areas of mahtematics such as integrable systems, 
mirror symmetry,  Langlands programme, Hodge theory, among others. We are interested on their link to Hodge theory. The work of 
Simpson~\cite{sim1, sim2}, Hausel~\cite{hau}, and Hausel \& Thaddeus~\cite{hath1, hath2} shall be particularly useful for our purposes.

\noindent There is a Morse function $f\: \sM^k(3,d) \to \bR$ defined by
\[
f(E,\varphi) = \frac{1}{2\pi}\lVert \varphi \rVert^2_{L_2} 
= \frac{i}{2\pi} \int_X \tr(\varphi \varphi^*) 
\]
applied to the moduli spaces of stable $k$-Higgs bundles $\sM^k(r,d)$. We study the stabilization of the cohomology groups of 
the critical submanifolds from this Morse function $f$, for the case of rank $r = 3$. The co-prime condition $(3,d) = 1$ implies 
that the moduli space $\sM^k(3,d)$ is smooth. A $k$-Higgs bundle or Higgs bundle with poles of order $k$, $(E,\varphi^k)$, is 
a Higgs bundle where the morphism $\varphi^k$ is twisted by $L_p$ $k$-times, where $p\in X$ is an arbitrary fixed point and 
$L_p = \sO_{X}(p)$ is its associated line bundle (local structure sheaf): 
\[
 \varphi^k\: E \to E\ox K\ox L_p^{\ox k} = E \ox K(k\cdot p). 
\]
According to Simpson~\cite{sim1} the critical points of $f$, are
{\em variations of the Hodge structure} (VHS), a decomposition of the form:
\begin{equation}
E = \bigoplus_{j=1}^n E_j \word{such that} 
\varphi \: E_j \to E_{j+1} \ox K \textmd{ for } 1 \leq j \leq n - 1.
\end{equation} 
for general rank $r$. In our case, for $\sM^{k}(3,d)$, there are three kind of variations of Hodge structure:

\begin{enumerate}[i.]
 \item 
 $(1,2)$-VHS:
\[
 \Big(
 E_1\oplus E_2,
 \left(
 \begin{array}{c c}
  0 & 0\\
  \phi & 0
 \end{array}
 \right)
 \Big)\in
 F_{d_1}^{k}\sse
 \sM^{k}(3,d).
\]
 \item 
 $(2,1)$-VHS:
\[
 \Big(
 E_2\oplus E_1,
 \left(
 \begin{array}{c c}
  0 & 0\\
  \phi & 0
 \end{array}
 \right)
 \Big)\in
 F_{d_2}^{k}\sse
 \sM^{k}(3,d).
\]
 \item 
 $(1,1,1)$-VHS:
\[
\Big(
L_1\oplus L_2\oplus L_3,
\left(
 \begin{array}{c c c}
     0     &     0     & 0\\ 
 \phi_{21} &     0     & 0\\
     0     & \phi_{32} & 0
 \end{array}
\right)
\Big)\in 
F_{m_1 m_2}^{k}\sse
\sM^{k}(3,d),
\]
\end{enumerate}
Here, $F_{d_1}^{k},\ F_{d_2}^{k}$ and $F_{m_1 m_2}^{k}$ denote the respective critical submanifolds of the moduli space $\sM^{k}(3,d)$. The first 
two, $F_{d_1}^{k}$ and $F_{d_2}^{k}$, are close related to the space $\sN_{\sg}(r_1,r_2,d_1,d_2)$, the moduli space of $\sg$-stable holomorphic 
triples of type $(\rr,\dd) = (r_1,r_2,d_1,d_2)$.

\noindent A holomorphic triple $T = (E_1,E_2,\phi)$ on $X$ consists of a pair of holomorphic vector bundles $E_1\to X$ and $E_2\to X$, 
of ranks $r_1, r_2$ and degrees $d_1,d_2$ respectively, and a holomorphic map $\phi\: E_2\to E_1$. The stability for triples depends on a 
parameter $\sg \in \bR$,  which gives a collection of moduli spaces $\sN_{\sg}(r_1,r_2,d_1,d_2)$ widely worked by several authors: \emph{e.g.}
\cite{
brgp,
bgg1,
bgg2,
ggm,
mos,
mov}.
The range of $\sg$ is an interval $[\sg_{m}, \sg_{M}] \sse \bR$ split by a finite number of critical values $\sg_c$. The reader may see 
Bradlow, Garc\'ia-Prada, Gothen~\cite{bgg1}, Mu\~noz, Oliveira, S\'anchez~\cite{mos}, or Mu\~noz, Ortega, V\'asquez-Gallo~\cite{mov} for the 
interval details. 

\noindent This paper works with a very particular framework. We study holomorphic triples on $X$ of the form $T = (\tE_1,\tE_2,\phi)$ with type 
$(2,1,\td_1,\td_2)$, where ranks $r_1 = 2,\ r_2 = 1$ and degrees $\td_1, \td_2$ are in terms of $(1,2)$-VHS described before: 
\[
 \Big(
 E_1\oplus E_2,
 \left(
 \begin{array}{c c}
  0 & 0\\
  \phi & 0
 \end{array}
 \right)
 \Big)\in
 F_{d_1}^{k}\sse
 \sM^{k}(3,d),
\]
where $\tE_1 = E_2 \ox K(kp)$, $\tE_2 = E_1$, $\phi\: E_1 \to E_2 \ox K(kp)$, and so the degrees become 
$\td_1 = \deg(\tE_1) = d_2 + 2(2g - 2 + k)$, $\td_2 = d_1$.

\noindent We study as well the holomorphic triples $T = (\tE_1,\tE_2,\phi)$ with type $(1,2,\td_1,\td_2)$, related to $(2,1)$-VHS of the form 
\[
 \Big(
 E_2\oplus E_1,
 \left(
 \begin{array}{c c}
  0 & 0\\
  \phi & 0
 \end{array}
 \right)
 \Big)\in
 F_{d_2}^{k}\sse
 \sM^{k}(3,d),
\]
where, in this case $\tE_1 = E_1 \ox K(kp)$, $\tE_2 = E_2$, $\phi\: E_2 \to E_1 \ox K(kp)$, and the degrees become 
$\td_1 = \deg(\tE_1) = d_1 + 2g - 2 + k$, $\td_2 = d_2$.

\noindent Finally, we also study $(1,1,1)$-VHS
\[
\Big(
L_1\oplus L_2\oplus L_3,
\left(
 \begin{array}{c c c}
     0     &     0     & 0\\ 
 \phi_{21} &     0     & 0\\
     0     & \phi_{32} & 0
 \end{array}
\right)
\Big)\in 
F_{m_1 m_2}^{k}\sse
\sM^{k}(3,d),
\]
and those are related to symmetric products of the form 
\[
\Sym^{m_1}(X)\times \Sym^{m_2}(X)\times \sJ^{d_3}(X)
\]
where $J^{d_3}(X)$ is the Jacobian of $X$, the moduli space of stable line bundles of degree $d_3$, and $m_1, m_2$ will be described below as the corresponding degrees of auxiliar bundles.

\noindent Our estimates are based on embeddings $\sM^k(3,d) \hookrightarrow \sM^{k+1}(3,d)$ defined by
\[
i_k\: 
\big[(E,\varphi^k)\big] 
\longmapsto 
\big[(E,\varphi^k \ox s_p)\big]
\]
where $0 \neq s_p \in H^0(X, L_p)$ is a nonzero fixed section of~$L_p = \sO_X(p)$. 

\noindent The paper is organized as follows: 
in section \ref{sec:1} we recall 
some basic facts about holomorphic 
triples, Higgs bundles and $k$-Higgs bundles; 
in section \ref{sec:2}, we present 
the effect of the embeddings on 
$\sg$-stable triples; 
in subsection \ref{ssec:2.1}, we show that the embeddings preserve $\sg$-stability,
in subsection \ref{ssec:2.2}, we discuss the effect of the embeddings 
considering the flip loci, and present an original result, 
the so-called \emph{``Roof Theorem''}: 

\begin{Th}[Theorem \ref{RoofTheorem}]
 There exists an embedding
 \[
 \tilde{i_k}: \tilde{\sN}_{\sg_c(k)} 
 \hookrightarrow 
 \tilde{\sN}_{\sg_c(k+1)}
 \]
 such that the following diagram commutes: 
 \begin{align*}
 \begin{xy}
(0,-20)*+{\sN_{\sg_c^{-}(k+1)}}="a";
(20,0)*+{\tilde{\sN}_{\sg_c(k+1)}}="b";
(10,-50)*+{\tilde{\sN}_{\sg_c(k)}}="c";
(40,-20)*+{\sN_{\sg_c^{+}(k+1)}}="d";
(-10,-70)*+{\sN_{\sg_c^{-}(k)}}="e";
(30,-70)*+{\sN_{\sg_c^{+}(k)}}="f";
{\ar@{-->}^(.45){\exists \tilde{i_k}} "c";"b" **\dir{--}};
{\ar^(.45){} "b";"a" **\dir{-}};
{\ar^(.45){i_k} "e";"a" **\dir{-}};
{\ar^(.45){} "c";"e" **\dir{-}};
{\ar_(.45){i_k} "f";"d" **\dir{-}};
{\ar^(.45){} "c";"f" **\dir{-}};
{\ar^(.45){} "b";"d" **\dir{-}};
 \end{xy}
 \end{align*}
where $\tilde{\sN}_{\sg_c(k)}$ is the blow-up of 
$
\sN_{\sg_c^{-}(k)} = 
\sN_{\sg_c^{-}(k)}(2,1,\td_1,\td_2)
$ 
along the flip locus $S_{\sg_c^{-}(k)}$ and, at the same time, represents the blow-up of
$
\sN_{\sg_c^{+}(k)} = 
\sN_{\sg_c^{+}(k)}(2,1,\td_1,\td_2)
$ 
along the flip locus $S_{\sg_c^{+}(k)}$.
\end{Th}

\noindent Here, $\sg_c(k) \in ]\sg_m,\sg_M[$ is a $\sg$-critical value depending on the parameter $k$, that lies in the interval mentioned above, where 
$\sg_m = \mu_1 - \mu_2 = \td_1/2 - \td_2$ and $\sg_M = 4\sg_m$.

\noindent In section \ref{sec:3} we present the cohomology main results for triples: 
in subsection \ref{ssec:3.1} appear some useful results about the cohomology 
of the symmetric product $\Sym^{k}(X)$. 
In subsection \ref{ssec:3.2} we present
the stabilization of the cohomology (Theorem \ref{CohomologyBlowUp}) for certain indices:

\begin{Th}[Theorem~\ref{CohomologyBlowUp}]
 There is an isomorphism
\[
 \tilde{i_k^*}\: 
 H^{j}(\tilde{\sN}_{\sg_c(k+1)},\mathbb{Z}) 
 \xrightarrow{\quad \cong \quad} 
 H^{j}(\tilde{\sN}_{\sg_c(k)},\mathbb{Z})\quad 
 \forall j \leqslant n(k)
\]
 at the blow-up level, where 
 $
 n(k) = \min \big\{ \td_1 - d_M - \td_2 - 1,\quad 2\big(\td_1 - 2\td_2 - (2g - 2)\big) + 1\big\}
 $.
\end{Th}

\noindent And hence, the cohomology stabilization of the moduli spaces of triples: 
\begin{Cor}[Corollary~\ref{CohomologyCriticalTriples}]
 There is an isomorphism
\[
 i_k^*\: 
 H^{j}(\sN_{\sg_c(k+1)},\mathbb{Z}) 
 \xrightarrow{\quad \cong \quad} 
 H^{j}(\sN_{\sg_c(k)},\mathbb{Z})\quad 
 \forall j \leqslant n(k)
\]
 where 
 $
 n(k)
 $ as above.
\end{Cor}

\noindent In subsection \ref{ssec:3.3} we show the stabilization 
of the $(1,2)$-VHS cohomology using the isomorphisms between them 
and the moduli spaces of triples:

\begin{Cor}[Corollary~\ref{(1,2)-VHS--Cohomology}]
 There is an isomorphism 
\[
 H^{j}(F_{d_1}^{k+1},\mathbb{Z}) 
 \xrightarrow{\quad \cong \quad} 
 H^{j}(F_{d_1}^{k},\mathbb{Z})
\]
 for all
 $
 j \leqslant \sg_H(k) - 2(\mu_1 - \mu) - 1
 $.
\end{Cor}
\noindent Here, $\sg_{H}(k) \in ]\sg_m,\sg_M[$ is a particular $\sg$-critical value depending on the parameter $k$: 
\[
 \sg_{H}(k) = \deg\big(K(k\cdot p)\big) = 2g - 2 + k.
\]

\noindent In subsection \ref{ssec:3.4} we show the analogous dual result for $(2,1)$-VHS:

\begin{Cor}[Corollary~\ref{(2,1)-VHS--Cohomology}]
 For $k$ large enough, there is an isomorphism 
 \[
 H^{j}(F_{d_2}^{k+1},\mathbb{Z}) 
 \xrightarrow{\quad \cong \quad} 
 H^{j}(F_{d_2}^{k},\mathbb{Z})  
 \]
 for all
 $
 j \leqslant \sg_H(k) - 4(\mu_2 - \mu) - 1.
 $
\end{Cor}

\noindent Finally, in subsection \ref{ssec:3.5} we described the cohomology for $(1,1,1)$-VHS and its relationship with the spaces 
$\Sym^{m_1}(X)\times \Sym^{m_2}(X)\times \sJ^{d_3}(X)$:

\begin{Cor}[Corollary~\ref{(1,1,1)-cohomology-iso}]
 There is an isomorphism
 \[
  H^{j}(F_{m_1 m_2}^{\infty},\bZ)
  \xrightarrow{\quad \cong \quad} 
  H^{j}(F_{m_1 m_2}^{k},\bZ)
 \]
 for all $j \leq \min \big(\bar{m}_1 + k, \bar{m}_2 + k\big) - 1$.
\end{Cor}

\section{Preliminary definitions} 
\label{sec:1}

Let $X$ be a compact connected Riemann surface of genus $g \geq 2$ and let $K = T^*X$ be the canonical line bundle of 
$X$. Note that, algebraically, $X$ is also a nonsingular complex projective algebraic curve.

\noindent The $k$-th symmetric product $\Sym^{k}(X)$ is a smooth projective variety of dimension $k\in \bN$, that could be interpretated as 
the moduli space of degree $k$ effective divisors. In other words, $\Sym^{k}(X) = X^{k}/S_{k}$, the symmetric product with quotient topology, 
is the quotient of $X^{k}$ the $k$-times cartesian product by the action of $S_{k}$ the $k$-symmetric group. Obviously $\Sym^{1}(X) = X$.

\begin{Def} 
 For a (smooth or holomorphic) vector bundle $E \to X$, we denote the \emph{rank} of~$E$ by
 $\rk(E) = r$ and the \emph{degree} of $E$ by $\deg(E) = d$. Its 
 \emph{slope} is defined to be
  \begin{equation}
    \mu(E) = \frac{\deg(E)}{\rk(E)} = \frac{d}{r}.
    \label{slope} 
  \end{equation} 
 A vector bundle $E \to X$ is called \emph{semistable}  if $\mu(F) \leq \mu(E)$ for any nonzero $F \subseteq E$.  Similarly, a vector bundle $E \to X$ is called \emph{stable}  if $\mu(F) < \mu(E)$ for any nonzero $F \subsetneq E$. Finally,  $E$ is called \emph{polystable} if it is the direct sum of stable subbundles, all of the same slope.
\end{Def}

\subsection{Holomorphic Triples} 
\label{ssec:1.1}

\begin{Def} 
  A \emph{holomorphic triple} on $X$ is a triple
  $T = (E_1, E_2, \phi)$ consisting of two holomorphic vector bundles
  $E_1 \to X$ and $E_2 \to X$ and a homomorphism $\phi \: E_2 \to E_1$,
  {\em i.e.} an element 
  $\phi \in H^0\big(\Hom(E_2,E_1)\big)$. 
\end{Def}

\begin{Def} 
 A \emph{homomorphism} from a triple $T' = (E'_1,E'_2,\phi')$ to
  another triple $T = (E_1,E_2,\phi)$ is a commutative diagram of the
  form:
  \[
  \begin{xy}
    (0,0)*+{E'_2}="a";
    (20,0)*+{E'_1}="b";
    (0,-20)*+{E_2}="c";
    (20,-20)*+{E_1}="d";
    {\ar@{->}^{\phi'} "a";"b"};
    {\ar@{->} "b";"d"};
    {\ar@{->} "a";"c"};
    {\ar@{->}^{\phi} "c";"d"};
  \end{xy}
  \]
  where the vertical arrows represent holomorphic maps.
\end{Def}

\begin{Def} 
 A triple $T' = (E_1',E_2',\phi')$ is a {\em subtriple} of $T = (E_1,E_2,\phi)$ if
 \begin{enumerate}[i.]
  \item 
  $E_j' \sse E_j$ is a coherent subsheaf for $j=1,2$
  \item 
  $\phi' = \phi|_{{}_{E_2'}}$, {\em i.e.} $\phi'$ is the restriction of $\phi$.
 \end{enumerate}
In other words, we get the commutative diagram
\[
  \begin{xy}
    (0,0)*+{E'_2}="a";
    (20,0)*+{E'_1}="b";
    (0,-20)*+{E_2}="c";
    (20,-20)*+{E_1}="d";
    {\ar@{->}^{\phi'} "a";"b"};
    {\ar@{->} "b";"d"};
    {\ar@{->} "a";"c"};
    {\ar@{->}^{\phi} "c";"d"};
  \end{xy}
\]
where the vertical arrows are injective inclusions this time. In such a case, we denote $T' \sse T$.

If $E_1' = E_2' = 0$ we call the subtriple $T' \sse T$ as {\em the trivial subtriple}.

$T'$ is a {\em non-trivial proper subtriple} if $0 \neq T' \subsetneq T$.
\end{Def}

\begin{Rmk} 
 For stability criteria, it will be enough to consider saturated subsheaves. In our case, since $X$ is a Riemann surface, 
 saturated subsheaves are precisely subbundles.
\end{Rmk}

\begin{Def} 
 A triple $T = (E_1,E_2,\phi)$ is {\em reducible} if there are direct sum decompositions $\displaystyle E_1  = \bigoplus_{i=1}^{n}E_{1i}$,
 $\displaystyle E_2  = \bigoplus_{i=1}^{n}E_{2i}$, and $\displaystyle \phi  = \bigoplus_{i=1}^{n}\phi_{i}$ such that 
 $\phi_{i} \in \Hom(E_{2i},E_{1i})$. In such a case, $T$ has a {\em direct sum decomposition}
 \[
  T = \bigoplus_{i=1}^{n} T_{i} \word{of subtriples} T_{i} = (E_{1i},E_{2i},\phi_{i}).
 \]
 If $T = (E_1,E_2,\phi)$ is not reducible, we say that $T$ is {\em irreducible}.
\end{Def}

\begin{Rmk} 
 We adopt Bradlow~and~Garc\'ia-Prada~\cite{brgp} convention that if $E_{2i} = 0$ or $E_{1i} = 0$ for some $i$, then $\phi_{i}$ is the zero map.
\end{Rmk}

\begin{Def} 
  $\sg$-Stability, $\sg$-Semistability and $\sg$-Polystability:
  \begin{enumerate}[i.]
  \item
  For any $\sg \in \bR$, the $\sg$-degree and the $\sg$-slope of 
  $T = (E_1,E_2,\phi)$ are defined as:
  \[
  \deg_\sg(T) = \deg(E_1) + \deg(E_2) + \sg \cdot \rk(E_2),
  \]
  and 
  \[
  \mu_\sg(T) = \frac{\deg_\sg(T)}{\rk(E_1) + \rk(E_2)}
  = \mu(E_1 \oplus E_2) 
  + \sg\, \frac{\rk(E_2)}{\rk(E_1) + \rk(E_2)} \,
  \]
  respectively.
  \item
  $T$ is then called $\sg$-stable [respectively, $\sg$-semistable] if 
  $\mu_\sg(T') < \mu_\sg(T)$ [respectively, $\mu_\sg(T') \leq \mu_\sg(T)$]
  for any proper subtriple $0 \neq T' \subsetneq T$.
  \item
  A triple is called $\sg$-polystable if it is the direct sum of
  $\sg$-stable triples of the same $\sg$-slope.
  \end{enumerate}
\end{Def}

Now we may use the following notation for moduli spaces of triples:
 \begin{enumerate}[i.]
 \item
 Denote $\rr = (r_1,r_2)$ and $\dd = (d_1,d_2)$, and then regard
 \[
 \sN_\sg = \sN_\sg(\rr,\dd) = \sN_\sg(r_1,r_2,d_1,d_2)
 \]
 as the moduli space of $\sg$-polystable triples $T = (E_1,E_2,\phi)$ 
 such that $\rk(E_j) = r_j$ and $\deg(E_j) = d_j$.
 \item
 Denote by $\sN^s_\sg = \sN^s_\sg(\rr,\dd)$ the open subspace of
 $\sg$-stable triples.
 \item
 Call $(\rr,\dd) = (r_1,r_2,d_1,d_2)$ the type of the triple
 $T = (E_1,E_2,\phi)$.
 \end{enumerate}

\noindent The moduli space of $\sg$-stable triples $\sN^s_\sg = \sN^s_\sg(\rr,\dd) = \sN^s_\sg(r_1,r_2,d_1,d_2)$ is formally constructed by 
Bradlow and Garc\'ia-Prada~\cite{brgp} using dimensional reduction. There is also a direct construction by Schmitt~\cite{schmitt} 
using Geometric Invariance Theory (GIT). The reader also may consult the work of Bradlow, Garc\'ia-Prada and
Gothen~\cite{bgg1}; Mu\~noz, Oliveira and S\'anchez~\cite{mos}; or Mu\~noz, Ortega and V\'azquez-Gallo~\cite{mov} for the following 
details.

\noindent There are certain necessary conditions on $\sg$ for $\sg$-stable triples to exist. For triples of type $(\rr,\dd) = (r_1,r_2,d_1,d_2)$, 
consider the slopes $\mu_j = \frac{d_j}{r_j}$ for $j = 1,2$ and define 
\begin{equation}\label{sg-m}
 \sg_m = \mu_1 - \mu_2,
\end{equation}
and
\begin{equation}\label{sg-M}
 \sg_M = 
 \left(
 1 + \frac{r_1 + r_2}{|r_1 - r_2|}
 \right)
 (\mu_1 - \mu_2), \word{for} r_1 \neq r_2,
\end{equation}

\begin{Th}[{\cite[Th.~6.1.]{brgp}}]
 The moduli space of $\sg$-stable triples $\sN_{\sg}^{s}(r_1,r_2,d_1,d_2)$ is a complex analytic variety, which is projective when $\sg \in \bQ$. 
 A necessary condition for $\sN_{\sg}^{s}(r_1,r_2,d_1,d_2)$ to be non-empty is 
 
 $
 0 \leq \sg_m \leq \sg \leq \sg_M, \word{if} r_1 \neq r_2,
 $
 
 $
 0 \leq \sg_m \leq \sg, \word{if} r_1 = r_2.
 $
\end{Th}

\begin{Rmk}
 If $\mu_1 = \mu_2$ and $r_1 \neq r_2$ then $\sg_m = \sg_M = 0$ and so, $\sN^{s}_{\sg}(r_1,r_2,d_1,d_2)$ is empty unless $\sg = 0$.
\end{Rmk}

\begin{Prop}[{\cite[Prop.~2.4.]{bgg1}}]\label{dualtriples}
 The $\sg$-stability of $T = (E_1,E_2,\phi)$ is equivalent to the $\sg$-stability of the dual triple $T^{*} = (E_2^*,E_1^*,\phi^*)$, where 
 $\phi^*$ represents the conjugate transpose of $\phi$. The map $T\mapsto T^*$ defines an isomorphism
 \[
  \sN_{\sg}^{s}(r_1,r_2,d_1,d_2)
  \cong
  \sN_{\sg}^{s}(r_2,r_1,-d_2,-d_1).
 \]
\end{Prop}

\noindent The last result is frequently used to restrict the study of triples to $r_1 \geq r_2$ and appeal to duality when $r_1 < r_2$. We shall use 
this duality result later to study and compare the cohomology of $(1,2)$-VHS and $(2,1)$-VHS.

\begin{Def}
 For triples of type $(r_1,r_2,d_1,d_2)$, the number $\sg \in [\sg_m, \infty[$ is a {\em critical value} if there exist integers 
 $r_1', r_2', d_1'$ and $d_2'$ such that 
 \[
  \sg =
  \frac{(r_1 + r_2)(d_1' + d_2')-(r_1' + r_2')(d_1 + d_2)}{r_1' r_2 - r_1 r_2'}
 \]
or equivalently
\[
 \frac{d_1 + d_2}{r_1 + r_2} + \frac{\sg \cdot r_2}{r_1 + r_2}
 =
 \frac{d_1' + d_2'}{r_1' + r_2'} + \frac{\sg \cdot r_2'}{r_1' + r_2'}
\]
with 
$0 \leq r_j' \leq r_j$,
$
(r_1',r_2',d_1',d_2')
\neq
(r_1,r_2,d_1,d_2)
$,
$(r_1,r_2) \neq (0,0)$
and 
$r_1' r_2 \neq r_1 r_2'$.

We denote $\sg = \sg_c$ if it is critical.

The number $\sg \in [\sg_m, \infty[$ is called {\em generic} if it is not critical.
\end{Def}

\begin{Prop}[{\cite[Prop.~2.6.]{bgg1}}]
 Fix the type $(r_1,r_2,d_1,d_2)$.
 \begin{enumerate}[i.]
  \item 
  The critical values $\sg_c$ form a discrete subset of the interval $[\sg_m, \infty[$.
  \item 
  If $r_1 \neq r_2$ the critical values $\sg_c$ are finite and lie in the interval $[\sg_m,\sg_M]$.
  \item 
  The stability criteria for two values of $\sg$ lying between two consecutive critical values are equivalent; thus, the moduli spaces 
  are isomorphic.
  \item 
  If $\sg$ is generic and $\GCD(r_2,r_1 + r_2, d_1 + d_2) = 1$, then $\sg$-semistability is equivalent to $\sg$-stability.  
 \end{enumerate}
\end{Prop}

\subsection{Higgs Bundles} 
\label{ssec:1.2}
 
\begin{Def}\label{hbdef} 
 A \emph{Higgs bundle} over $X$ is a pair $(E, \varphi)$ where 
 $E \to X$ is a holomorphic vector bundle and $\varphi\: E \to E \ox K$
 is an endomorphism of $E$ twisted by~$K$, which is called a 
 \emph{Higgs field}. Note that $\varphi \in H^0(X;\End(E) \ox K)$.
\end{Def}

\begin{Def}\label{stablehb} 
 A subbundle $F \subseteq E$ is said to be \emph{$\varphi$-invariant} if $\varphi(F) \subseteq F \ox K$. A Higgs bundle is said to be 
 \emph{semistable} [respectively, \emph{stable}] if $\mu(F) \leq \mu(E)$ [respectively, $\mu(F) < \mu(E)$] for any nonzero 
 $\varphi$-invariant subbundle $F \subseteq E$
 [respectively, $F \subsetneq E$]. Finally, $(E,\varphi)$ is called \emph{polystable} if it is the direct sum of stable
 $\varphi$-invariant subbundles, all of the same slope.
\end{Def}

\noindent Fixing the rank $\rk(E) = r$ and the degree $\deg(E) = d$ of a Higgs bundle $(E,\varphi)$, the isomorphism classes of polystable bundles are
parametrized by a quasi-projective variety: the moduli space $\sM(r,d)$. Constructions of this space can be found in the work of
Hitchin~\cite{hit2}, using gauge theory, or in the work of Nitsure~\cite{nit}, using algebraic geometry methods.

{

\noindent Hitchin~\cite{hit2} works with the \emph{Yang-Mills self-duality equations} (SDE)
 \begin{equation} 
 \left\{
 \begin{array}{c c c}
  F_A + [\varphi, \varphi^{*}] & = & 0 \\
                         &   &   \\
  \bar{\partial}_A \varphi  & = & 0,
 \end{array}
 \right.
 \label{eq:YM} 
 \end{equation}
where $\varphi \in \Omega^{1,0}\big(X, \End(\sE)\big)$ is a complex auxiliary 
field and $F_A$ is the curvature of a connection $d_A$ which is compatible with $\bar{\partial}_A$, the holomorphic structure of the bundle $E = (\sE, \bar{\partial}_A)$, where $\sE$ is a smooth complex bundle of rank $\rk(\sE)=2$ and degree $\deg(\sE)=1$. Hitchin calls $\varphi$ {\em Higgs~field}, because it shares a lot of the physical and gauge properties of those of the Higgs boson. Here, $\varphi^{*}$  denotes the adjoint of $\varphi$ with respect to the hermitian metric on $E$,\footnote{By Hitchin~\cite{hit2}, there is a hermitian metric on E.} and $[\cdot,\cdot]$ denotes the extension of the Lie bracket to Lie algebra-valued forms.

\noindent The set of solutions
 $$
 \beta(\sE) = 
 \{
 (\bar{\partial}_A, \varphi)| \word{solution of}(\ref{eq:YM})
 \}
 \subseteq
 \sA^{0,1}(\sE) \times \Omega^{1,0}\big(X,\End(\sE)\big)
 $$
where $\sA^{0,1}(\sE)$ denotes the space of holomorphic structures on $\sE$,
$\Omega^{1,0}\big(X,\End(\sE)\big)$ denotes $(1,0)$-forms of $X$ with values on $\End(\sE)$,
 and the collection
 $$
 \beta_{ps}(2,1) = 
 \{
 \beta(\sE)| \sE \word{polystable,} \rk(\sE) = 2, \deg(\sE) = 1 
 \},
 $$
allow Hitchin to construct the Moduli space of solutions to SDE~(\ref{eq:YM})
  $$
  \sM^{YM}(2,1) = \beta_{ps}(2,1) / \gG^{\bC},
  $$
and
  $$
   \sM^{YM}_{s}(2,1) = \beta_{s}(2,1) / \gG^{\bC}
   \subseteq
   \sM^{YM}(2,1),
  $$
the moduli space of stable solutions to SDE~(\ref{eq:YM}), where $\gG^{\bC}$ represents the complex gauge group, which acts by conjugation on $\beta_{ps}(2,1)$ and $\beta_{s}(2,1)$.

\begin{Rmk}
 Since $\GCD(2,1) = 1$, then $\beta_{ps}(2,1) = \beta_{s}^(2,1)$ and so 
  $$
  \sM^{YM}(r,d) = \sM^{YM}_{s}(r,d).
  $$
\end{Rmk}

\noindent Using definition~\ref{hbdef}, and the notion of stability~\ref{stablehb}, Hitchin~\cite{hit2} presents an alternative algebro-geometric construction of the moduli space of polystable Higgs bundles:
   $$
   \sM^{H}(2,1) = \{(E,\varphi)|\ E \word{polystable} \} / \gG^{\bC}
   $$
and the subspace
   $$
   \sM^{H}_{s}(2,1) = 
   \{(E,\varphi)|\ E \word{stable} \} / \gG^{\bC}
   \subseteq
   \sM^{H}(2,1),
   $$
of stable Higgs bundles.

\begin{Rmk}
  Again, $\GCD(2,1) = 1$ implies $\sM^{H}(2,1) = \sM^{H}_{s}(2,1)$. 
\end{Rmk}

\noindent Finally, Hitchin~\cite{hit2} concludes:

 \begin{Th}\textnormal{\cite{hit2}}
  There is a homeomorphism of topological spaces
  $$
   \sM^{H}(2,1)
   \cong
   \sM^{YM}(2,1).
   \QEDA
  $$
 \end{Th}

\noindent Because of the last homeomorphism, from now on it will be enough to denote 
$
\sM(2,1)
=
\sM^{H}(2,1)
   \cong
\sM^{YM}(2,1)
$ 
for brief. Hitchin~\cite{hit2} computes the real dimension of the moduli space of stable rank two pairs $(E,\varphi)$:

\begin{Th}[{\cite[Th.~5.8.]{hit2}}]\label{hit5.8.}
 Let $X=\Sg_g$ be a compact Riemann surface of genus $g > 1$. The moduli space $\sM(2,1)$ of all stable pairs $(E,\varphi)$, where $E\to X$ 
 is a rank two holomorphic vector bundle of degree one, and $\varphi$ is a trace free holomorphic section of $\End(E)\ox K$, is a smooth real 
 manifold of dimension
 \[
  \dim_{\bR}\big(\sM(2,1)\big) = 12(g - 1).
 \]
\end{Th}

\begin{Cor}
 The space $\sM(2,1)$ is a quasi--projective variety of complex dimension
 $$
 \dim_{\bC}\big(\sM(2,1)\big) = 3(2g - 2).
 $$ 
\end{Cor}

\noindent Nitsure~\cite{nit} constructs the moduli space of Higgs bundles of general rank $r$ and degree $d$ using Geometric Invariant Theory (GIT), and computes its dimension:

\begin{Th}[{\cite{nit}}]
 The space $\sM(r,d)$ is a quasi--projective variety of complex dimension
 $$
 \dim_{\bC}\big(\sM(r,d)\big) = (r^2-1)(2g - 2).
 $$ 
\end{Th}

\begin{Rmk}
 Note that the result of Nitsure~\cite{nit} coincides with the result of Hitchin~\cite{hit2} for rank two Higgs bundles.
\end{Rmk}

\noindent Simpson~\cite{sim2} calls the pair $(E,\varphi)$ as {\em Higgs bundle}. His work contributes generalizing Higgs bundles to higher dimensions and proving an analogous proposition to Theorem~\ref{hit5.8.} for general rank, with the same notion of stability in mind, considering the moduli space of Higgs bundles as the quotient
   $$
   \sM^{H}(r,d) = \{(E,\varphi)|\ E \word{polystable} \} / \gG^{\bC}
   $$
and the subspace
   $$
   \sM^{H}_{s}(r,d) = 
   \{(E,\varphi)|\ E \word{stable} \} / \gG^{\bC}
   \subseteq
   \sM^{H}(r,d),
   $$
of stable Higgs bundles.

\begin{Rmk}
  Once again, $\GCD(r,d) = 1$ implies $\sM^{H}(r,d) = \sM^{H}_{s}(r,d)$. See 
  Simpson~\cite{sim2} for details.
\end{Rmk}

 \begin{Th}[{\cite[Prop.~1.5]{sim2}}]
  There is a homeomorphism of topological spaces
  $$
   \sM^{H}(r,d)
   \cong
   \sM^{YM}(r,d).
   \QEDA
  $$
 \end{Th}
}

\noindent Also for general rank, we denote 
$
\sM(r,d)
=
\sM^{H}(r,d)
   \cong
\sM^{YM}(r,d)
$ 
for brief, or even $\sM = \sM(r,d)$ when the rank $r$ and the degree $d$ are clear. Recall that we are considering the co-prime case $\GCD(r,d) = 1$, 
in order for $\sM = \sM(r,d)$ to be a smooth variety. An important feature of $\sM(r,d)$ is that it carries an action 
of~$\bC^*$: $z \cdot (E, \varphi) = (E, z \cdot \varphi)$. According to Hitchin~\cite{hit2}, $(\sM,I,\Om)$ is a K\"ahler manifold, where $I$ is 
its complex structure and $\Om$ its corresponding
K\"ahler form. Furthermore, $\bC^*$ acts on $\sM$ biholomorphically with respect to the complex structure $I$ by the aforementioned action, 
where the K\"ahler form $\Om$ is invariant under the induced action $e^{i\theta} \cdot (E, \varphi) = (E, e^{i\theta} \cdot \varphi)$ of the
circle $\bS^1 \subseteq \bC^*$. Besides, this circle action is Hamiltonian, with proper moment map $f \: \sM \to \bR$ defined by:
\begin{equation}
f(E, \varphi) = \frac{1}{2\pi} \|\varphi\|_{L^2}^2
= \frac{i}{2\pi} \int_X \tr(\varphi \varphi^*)
\label{momentum_map_intro} 
\end{equation}
where $\varphi^*$ is the adjoint of $\varphi$ with respect to the hermitian metric on~$E$, and $f$ has finitely many critical values.

\noindent There is another important fact mentioned by Hitchin~\cite{hit2} (see the original version in the work of Frankel~\cite{fra}, and its application to Higgs bundles 
in~\cite{hit2}): the critical points of $f$ are exactly the fixed points of the circle action on~$\sM$.

\noindent If $(E, \varphi) = (E, e^{i\theta}\varphi)$ and $\varphi = 0$, then the critical value is $c_0 = 0$. The corresponding critical submanifold is
$F_0 = f^{-1}(c_{0}) = f^{-1}(0) = \sN$, the moduli space of stable bundles (see Hitchin~\cite{hit2}, Simpson~\cite{sim2}, or Bradlow, 
Garc\'ia-Prada, Gothen~\cite{bgg2} for details). On the other hand, when $\varphi \neq 0$, there is a type of
algebraic structure for Higgs bundles introduced by Simpson~\cite{sim1, sim2}: a \emph{variation of Hodge structure}, or simply a \emph{VHS}, 
for a Higgs bundle $(E, \varphi)$ is a decomposition:
\begin{equation}
E = \bigoplus_{j=1}^n E_j \word{such that} 
\varphi \: E_j \to E_{j+1} \ox K \textmd{ for } 1 \leq j \leq n - 1.
\label{VHS} 
\end{equation} 

\noindent It has been proved by Simpson~\cite{sim1} that the fixed points of the circle action on $\sM(r,d)$, and so, the critical points of $f$, are
these variations of the Hodge structure VHS, where the critical values $c_\la = f(E,\varphi)$ will depend on the degrees $d_j$ of the
components $E_j \subseteq E$, and $\lambda$ denotes the index of the critical point for the Morse-Bott function $f$. By Morse theory, we can stratify $\sM$ in such a way that there is a non-minimal critical submanifold
$F_\la = f^{-1}(c_\la)$ for each nonzero critical value $0 \neq c_\la = f(E,\varphi)$ where $(E,\varphi)$ represents a fixed point
of the circle action, or equivalently, a VHS. We then say that $(E,\varphi)$ is a $\big(\rk(E_1),\dots,\rk(E_n)\big)$-VHS.

\noindent The Morse indexes of the critical submanifolds of the moduli space of stable Higgs bundles $\sM(r,d)$ for general rank $r$ were 
calculated by Bradlow~et~al.~\cite{bgg2}:

\begin{Prop}[{\cite[Prop.~3.10.]{bgg2}}]
Let $(E,\varphi)$ be a stable Higgs bundle which corresponds to a critical point of $f$. Then the Morse index of the corresponding critical submanifold $(E,\varphi)\in F_{\lambda}$ is
\[
 {index}(F_{\la}) = 
 2 \sum_{k > 0} \dim \Big(\mathbb{H}^1\big(C_{k}^{\bullet}(E,\varphi) \big) \Big)
\]
where
\[
 \dim \Big(\mathbb{H}^1\big(C_{k}^{\bullet}(E,\varphi) \big) \Big)
 =
 -\chi\big(C_{k}^{\bullet}(E,\varphi) \big)
\]
and $C_{k}^{\bullet}(E,\varphi)$ is the deformation complex of the pair $(E,\varphi)$. \QEDA
\end{Prop}

\begin{Prop}[{\cite[Prop.~3.12.(2)]{bgg2}}]
 For $\sM(r,d)$
 \[
 {index}(F_{\la}) \geq 
 (r - 1)(2g - 2)
 \]
 for every non-minimal critical submanifold $F_{\lambda} \sse \sM(r,d)$.\QEDA
\end{Prop}

\begin{Prop}[{\cite[Prop.~3.14.]{bgg2}}]
 Let $F_{0} \sse \sM(r,d)$ be the set of local minima. Then 
 \[
  F_{0} = 
  \left\{ 
  (E,\varphi) \in \sM(r,d) |\ \varphi = 0
  \right\}.
 \]
Hence, $F_{0}$ coincides with $\sN(r,d)$, the moduli space of semistable bundles of rank $r$ and degree $d$, which equals the subvariety $\sN^{s}(r,d) \sse \sN(r,d)$ corresponding to stable bundles if $\GCD(r,d) = 1$. \QEDA
\end{Prop}

\subsection{k-Higgs Bundles} 
\label{ssec:1.3}

\begin{Def} 
  Fix a point $p \in X$, and let $L_p = \sO_X(p)$ be the associated line bundle to the divisor $p \in \Sym^1(X) = X$. A
  \emph{$k$-Higgs bundle} (or \emph{Higgs bundle with poles of order~$k$}) is a pair $(E,\varphi^k)$ where:
  \[
  E \xrightarrow{\ \varphi^k\ } E \ox K \ox L_p^{ \ox k} = E \ox K(k\cdot p)
  \]
  and where the morphism $\varphi^k \in H^0\big(X, \End(E) \ox K(k\cdot p)\big)$ 
  is what we call  a \emph{Higgs field with poles of order~$k$}. The moduli space of $k$-Higgs bundles of rank~$r$ and degree~$d$ 
  is denoted by $\sM^k(r,d)$. For simplicity, we will suppose that $\GCD(r,d) = 1$, and so, $\sM^k(r,d)$ will be smooth.
\end{Def}
\begin{Rmk}
So far, everything we have said for Higgs bundles and the moduli space $\sM(r,d)$ also hold for $k$-Higgs bundles and the moduli spaces $\sM^k(r,d)$. 
\end{Rmk}

\noindent There is a new tool for $k$-Higgs bundles: an embedding of the form
\begin{equation}
 i_k \: \sM^k(r,d) \to \sM^{k+1}(r,d) 
 \: [(E,\varphi^k)] \longmapsto [(E,\varphi^k \ox s_p)]
 \label{eq:twist-embed} 
\end{equation}
where $0 \neq s_p \in H^0(X, L_p)$ is a nonzero fixed section of~$L_p$.

\noindent When the rank is $r = 2$ or $r = 3$, the map $i_k$ induces embeddings of the form
\[
F^k_{\lambda} 
\xrightarrow{\quad i_k \quad} 
F^{k+1}_{\lambda}\quad \forall \lambda,
\]
for non-minimal\footnote{For stable pairs in $F^k_ 0 = \sN_k$, the embeddings are trivial. Cf.~Hausel~\cite[Ch.~3.~Sec.~3.4.]{hau}.} critical 
submanifolds $F^k_{\la}$, where $\la$ is the Morse index of the submanifold. 

\noindent For $\sM^k(2,1)$, when $r = 2$, the Morse index is $\lambda = 2(g + 2d_1 - 2) + k$, which depends just on the parameter 
$d_1 \in ]\frac{1}{2},g - \frac{1}{2} + \frac{k}{2}[\cap \bZ$ since $d = \deg(E) = 1$ (co-prime case $\GCD(r,d) = 1$), $g \geq 2$ is fixed, 
and $k$ is the order of the pole.

\noindent Hence, we may index the $(1,1)$-critical submanifolds as $F^{k}_{d_1}$, and the embeddings are well defined:
\[
  \begin{xy}
    (0,0)*+{F^{k}_{d_1}\cong \Sym^{\bar{d_1}+k}(X)}="a";
    (65,0)*+{\Sym^{\bar{d_1}+k+1}(X)\cong F^{k+1}_{d_1}}="b";
    {\ar@{->} "a";"b"};
    (5,-10)*+{i_k\: D}="c";
    (60,-10)*+{D + p}="d";
    {\ar@{|->} "c";"d"};
  \end{xy}
\]
where $D\in \Sym^{\bar{d_1}+k}(X)$ is a divisor and $\bar{d_1} = 2g - 2d_1 - 1$ for simplicity. The reader may see 
Bento~\cite{ben}, Hausel~\cite{hau}, Hausel~and~Thaddeus~\cite{hath1, hath2} or Hitchin~\cite{hit2} for details.

\noindent For $\sM^k(3,d)$, when $r = 3$, we have three types of critical submanifolds. For $(1,2)$-critical submanifolds $F^k_{\lambda}$, 
the Morse index is given by $\lambda = 2(3d_1 - d + 2g - 2 + k)$ 
where once again $d_1 = \deg(E_1)$ is the degree of the maximal destabilizing line bundle $E_1\subseteq E$, and so, we are in a very similar situation 
than before. Without lost of generality, we may pick the index $d_1$ for the $(1,2)$-critical submanifolds, and the embeddings become
\[
  \begin{xy}
    (0,0)*+{i_k\: F^{k}_{d_1}}="a";
    (65,0)*+{F^{k+1}_{d_1}}="b";
    {\ar@{->} "a";"b"};
    (0,-10)*+{(E,\varphi^k) = 
    \Big(
 E_1\oplus E_2,
 \left(
 \begin{array}{c c}
  0 & 0\\
  \phi^k & 0
 \end{array}
 \right)
 \Big)
    }="c";
    (80,-10)*+{(E,\varphi^k\ox s_p) = 
    \Big(
 E_1\oplus E_2,
 \left(
 \begin{array}{c c}
  0 & 0\\
  \phi^k \ox s_p & 0
 \end{array}
 \right)
 \Big)
    }="d";
    {\ar@{|->} "c";"d"};
  \end{xy}
\]
where $\phi^k\: E_1 \to E_2\ox K(kp)$ and $\frac{d}{3} < d_1 < \frac{d}{3} + g - 1 + \frac{k}{2}$ as we shall see below 
(see Bento~\cite{ben}, Gothen~\cite{got} or Z-R~\cite{z-r} for interval details). Moreover
\[
 (\phi^k \ox s_p)(E_1) \subseteq
 \phi^k(E_1) \ox L_p \subseteq
 E_2 \ox K \ox L_p^{\ox k} \ox L_p = 
 E_2 \ox K \ox L_p^{\ox k+1}
\]
and therefore
\[
 i_k(F_{d_1}^{k}) \subseteq F_{d_1}^{k+1}.
\]
\begin{Lem}[{\cite[Lema~2.3.1.]{ben}}]\label{bento2.3.1.}
Let $(E,\varphi^k)\in F_{d_1}^{k}$ be a $k$-Higgs bundle of the form
\[
(E,\varphi^k) = 
 \Big(
  E_1\oplus E_2,
  \left(
  \begin{array}{c c}
     0   & 0\\
  \phi^k & 0
 \end{array}
 \right)
 \Big).
\]
Hence, $(E,\varphi^k)$ is stable if and only if the holomorphic triple $T = (E_2\ox K(k\cdot p),E_1,\phi^k)$ is $\sg_{H}$-stable where 
$\sg_H = \sg_H(k) = \deg\big(K(k\cdot p)\big) = 2g - 2 + k$.
\end{Lem}

\begin{proof}
The pair $(E,\varphi^k)$ is stable if and only if the holomorphic chain
\[
\mathcal{C}\: \sE_1 = E_1 \to \sE_2 = E_2\ox K(k\cdot p)
\]
is $\alpha = \big(\sg_H(k),0\big)$ stable, which means that any proper subchain $\mathcal{C}'\: \sE_1'\to \sE_2'$ has $\alpha$-slope 
$\mu_{\alpha}(\mathcal{C}') < \mu_{\alpha}(\mathcal{C})$; considering a subbundle $\sE_1'\sse \sE_1 = E_1$ with degree $\deg(\sE_1') = d'_1$ 
and a subbundle $\sE_2' \sse \sE_2$ with degree $\deg(\sE_2') = d'_2$, then $\sE_2'\ox \big(K(k\cdot p) \big)^{*}\sse E_2$ is a 
subbundle with degree $\deg\big(\sE_2'\ox K(k\cdot p)^{*}\big) = d_2' - r_2'(2g - 2 + k)$, and we have
\[
 (E,\varphi^k) \word{stable} 
 \Longleftrightarrow 
 \frac{d_1' + d_2' - r_2'(2g - 2 + k)}{r_1' + r_2'} < \frac{d_1 + d_2}{3}
\]
where $r_j' = \rk(\sE_j')$.

\noindent On the other hand, suppose 
$(E,\varphi^k) = 
\Big(
  E_1\oplus E_2,
  \left(
  \begin{array}{c c}
     0   & 0\\
  \phi^k & 0
 \end{array}
 \right)
\Big)
$
is stable. The holomorphic triple $T = \big(E_2\ox K(k\cdot p), E_1, \phi^k\big)$ is $\sg$-stable if and only if any proper subtriple 
$T' = \big(E'_2\ox K(k\cdot p), E'_1, (\phi^k)'\big)$ satisfies 
$
 \mu_{\sg}(T') < \mu_{\sg}(T)
 \Leftrightarrow 
$
\[
 \frac{\deg(E_1') + \deg\big(E_2'\ox K(k\cdot p)\big) + \rk(E_1') \cdot \sg}{\rk(E_1') + \rk(E_2')} < 
\]
\[ 
 \frac{\deg(E_1) + \deg\big(E_2\ox K(k\cdot p)\big) + \rk(E_1)\cdot \sg}{\rk(E_1) + \rk(E_2)}
 \Leftrightarrow 
\]
\[
 \frac{d_1' + d_2' + r_2'(2g - 2 + k) + r_1' \sg}{r_1' + r_2'} < \frac{d_1 + d_2 + 2(2g - 2 + k) + \sg}{1 + 2}
\]
\[
 \Leftrightarrow 
 \frac{d_1' + d_2' + r_2'\sg_H(k) + r_1' \sg}{r_1' + r_2'} < \frac{d_1 + d_2 + 2\sg_H(k) + \sg}{3}.
\]
Since $(E,\varphi^k)$ is stable, it is enough to take
\[
 \frac{r_2'\cdot \sg_H(k) + r_1'\cdot \sg}{r_1' + r_2'} = \frac{2\cdot \sg_H(k) + \sg}{3}
 \Leftrightarrow
 r_2' \sg_H(k) - 2 r_1' \sg_H(k) = r_2' \sg - 2 r_1' \sg
 \]
 \[
 \Leftrightarrow
 (r_2' - 2r_1') \sg_H(k) = (r_2' - 2r_1') \sg
 \Leftrightarrow
 \sg_H(k) = \sg
\]
and so, the triple $T = (E_2\ox K(k\cdot p),E_1,\phi^k)$ is $\sg_H(k)$-stable.
\end{proof}

\begin{Rmk}
 Note that, at the last part of the proof above, the equality $r_2' = 2r_1'$ does not hold, because of the stability of $(E,\varphi^k)$.
\end{Rmk}

\begin{Prop}[{\cite[Proposi\c{c}\~ao~2.3.2.]{ben}}]\label{bento2.3.2.}
For each $d_1 \in  \big]\frac{d}{3},\frac{d}{3}+\frac{\sg_{H}(k)}{2}\big[\cap \bZ$ there is a $(1,2)$ critical submanifold of $\sM^k(3,d)$ of the form
\[
F_{d_1}^{k} = 
\big\{
(E,\varphi^k) = 
 \Big(
  E_1\oplus E_2,
  \left(
  \begin{array}{c c}
     0   & 0\\
  \phi^k & 0
 \end{array}
 \right)
 \Big)\:
 d_1 = \deg(E_1),\ 
 \rk(E_1) = 1,\
 \rk(E_2) = 2
\big\}.
\]
Furthermore, there is an isomorphism
\[
F_{d_1}^{k} \cong \sN_{\sg_H(k)}(2,1,d - d_1 + 2\sg_H(k),d_1)
\]
with the moduli space of $\sg_{H}(k)$-stable triples of this type.
\end{Prop}

\begin{proof}
The isomorphism is given by:
\[
  \begin{xy}
    (0,0)*+{F^{k}_{d_1}}="a";
    (65,0)*+{\sN_{\sg_H(k)}(2,1,d - d_1 + 2\sg_H(k),d_1)}="b";
    {\ar@{->} "a";"b"};
    (0,-10)*+{(E,\varphi^k) = 
    (
 E_1\oplus E_2,
 \left(
 \begin{array}{c c}
  0 & 0\\
  \phi^k & 0
 \end{array}
 \right)
 )
    }="c";
    (80,-10)*+{(E_2 \ox K(k\cdot p), E_1,\phi^k)}="d";
    {\ar@{|->} "c";"d"};
  \end{xy}
\]
where $\sg_{H}(k) = 2g - 2 + k$ as above.

\noindent In general, for the critical values $\sg_{c}$, we know that the interval is $\sg_{m} \leq \sg_{c} \leq \sg_{M}$
where 
\[
\sg_{m} = \mu_2 - \mu_1 =
\frac{\deg\big(E_2\ox K(k\cdot p)\big)}{r_2} - \frac{\deg(E_1)}{r_1} = 
\frac{d - d_1 + 2\sg_H(k)}{2} - d_1
\]
and 
\[
\sg_{M} = \Big(1 + \frac{r_2 + r_1}{|r_2 - r_1|} \Big)(\mu_2 - \mu_1) = 
4 \sg_{m} = 2\big(d - 3d_1 +2\sg_H(k)\big)
\]
(see \cite{bgg1}). So, in particular we have
\[
\sg_{H}(k) = 2g - 2 + k > 
\sg_{m} = \frac{d - d_1 + 2\sg_{H}(k)}{2} - d_1
\Longleftrightarrow
d_1 > \frac{d}{3}.
\]
On the other hand, we have
\[
\sg_{H}(k) < \sg_{M} = 
2\big(d - 3d_1 +2\sg_{H}(k)\big)
\Longleftrightarrow
d_1 < \frac{d}{3} + \frac{\sg_H(k)}{2}.
\]
Therefore,
\[
d_1\in 
\Big]
\frac{d}{3},
\frac{d}{3} + \frac{\sg_H(k)}{2}
\Big[\cap \bZ.
\]
\end{proof}

\begin{Rmk}
 In general, for the critical values $\sg_{c}$, the interval $[\sg_{m}, \sg_{M}]$ is closed. Nevertheless, for our particular case of 
 interest $\sg_m < \sg_H(k) < \sg_M$, so the interval will be open.
\end{Rmk}

\noindent For $(2,1)$-critical submanifolds $F^k_{\la} = F^k_{d_2}$, the Morse index is given by $\la = 2(3d_2 - 2d + 2g - 2 + k)$;
here $d_2 = \deg(E_2)$ is the degree of the maximal destabilizing bundle $E_2\subseteq E$ of rank two this time, and so, we are in 
a very similar situation than before:
\[
  \begin{xy}
    (0,0)*+{i_k\: F^{k}_{d_2}}="a";
    (65,0)*+{F^{k+1}_{d_2}}="b";
    {\ar@{->} "a";"b"};
    (0,-10)*+{(E,\varphi^k) = 
    \Big(
 E_2\oplus E_1,
 \left(
 \begin{array}{c c}
  0 & 0\\
  \phi^k & 0
 \end{array}
 \right)
 \Big)
    }="c";
    (80,-10)*+{(E,\varphi^k\ox s_p) = 
    \Big(
 E_2\oplus E_1,
 \left(
 \begin{array}{c c}
  0 & 0\\
  \phi^k \ox s_p & 0
 \end{array}
 \right)
 \Big)
    }="d";
    {\ar@{|->} "c";"d"};
  \end{xy}
\]
with $\phi^k\: E_2 \to E_1\ox K(kp)$ and $\frac{2d}{3} < d_2 < \frac{2d}{3} + g - 1 + \frac{k}{2}$ instead (see Bento~\cite{ben}, 
Gothen~\cite{got} or Z-R~\cite{z-r}). Furthermore,
\[
 (\phi^k \ox s_p)(E_2) \subseteq
 \phi^k(E_2) \ox L_p \subseteq
 E_1 \ox K \ox L_p^{\ox k} \ox L_p = 
 E_1 \ox K \ox L_p^{\ox k+1}
\]
and hence
\[
 i_k(F_{d_2}^{k}) \subseteq F_{d_2}^{k+1}.
\]

\begin{Lem}[{\cite[Lema~2.3.5.]{ben}}]\label{bento2.3.5.}
Let $(E,\varphi^k)\in F_{d_2}^{k}$ be a $k$-Higgs bundle of the form
\[
(E,\varphi^k) = 
 \Big(
  E_2\oplus E_1,
  \left(
  \begin{array}{c c}
     0   & 0\\
  \phi^k & 0
 \end{array}
 \right)
 \Big).
\]
Hence, $(E,\varphi^k)$ is stable if and only if the holomorphic triple $T = (E_1\ox K(k\cdot p),E_2,\phi^k)$ is $\sg_{H}$-stable where 
$\sg_H = \sg_H(k) = \deg\big(K(k\cdot p)\big) = 2g - 2 + k$.
\end{Lem}

\begin{proof}
The proof is very similar to that presented for the $(1,2)$-case in Lemma~\ref{bento2.3.1.}.
\end{proof}

\begin{Prop}[{\cite[Proposi\c{c}\~ao~2.3.6.]{ben}}]\label{bento2.3.6.}
For each $d_2 \in  \big]\frac{2d}{3},\frac{2d}{3}+\frac{\sg_H(k)}{2}\big[\cap \bZ$ there is a $(2,1)$ critical submanifold of $\sM^k(3,d)$ 
of the form
\[
F_{d_2}^{k} = 
\big\{
(E,\varphi^k) = 
 \Big(
  E_2\oplus E_1,
  \left(
  \begin{array}{c c}
     0   & 0\\
  \phi^k & 0
 \end{array}
 \right)
 \Big)\:
 d_2 = \deg(E_2),\ 
 \rk(E_2) = 2,\
 \rk(E_1) = 1
\big\}.
\]
Furthermore, there is an isomorphism
\[
F_{d_2}^{k} \cong \sN_{\sg_H(k)}(1,2,d - d_2 + \sg_H(k),d_2)
\]
with the moduli space of $\sg_{H}(k)$-stable triples.
\end{Prop}

\begin{proof}
In this case, the isomorphism is given by:
\[
  \begin{xy}
    (0,0)*+{F^{k}_{d_2}}="a";
    (65,0)*+{\sN_{\sg_H(k)}(1,2,d - d_2 + \sg_H(k),d_2)}="b";
    {\ar@{->} "a";"b"};
    (0,-10)*+{(E,\varphi^k) = 
    (
 E_2\oplus E_1,
 \left(
 \begin{array}{c c}
  0 & 0\\
  \phi^k & 0
 \end{array}
 \right)
 )
    }="c";
    (80,-10)*+{(E_1 \ox K(k\cdot p), E_2,\phi^k)}="d";
    {\ar@{|->} "c";"d"}.
  \end{xy}
\]
The rest of the proof is very similar to the $(1,2)$-case presented in Proposition~\ref{bento2.3.2.}.
\end{proof}

\noindent Finally, we consider the $(1,1,1)$-critical submanifolds of the form
\[
F^k_{\lambda} =
F_{d_1 d_2 d_3}^{k} =
\Bigg\{
\Big(
L_1\oplus L_2\oplus L_3,
\left(
 \begin{array}{c c c}
     0       &     0     & 0\\ 
 \phi^k_{21} &     0     & 0\\
     0     & \phi^k_{32} & 0
 \end{array}
\right)
\Big)
\Bigg\}\sse
\sM^{k}(3,d),
\]
where $L_j \sse E$ is a line bundle for $j = \{1,2,3\}$, we denote $d_j = \deg(L_j)$ and so, the degree of $E\to X$ could be write as 
$\deg(E) = d = d_1 + d_2 + d_3$. Using the fact that $d_3 = d - d_1 - d_2$ and considering auxiliar bundles 
$M_{j} = L_j^{*}\ox L_{j+1} \ox K(k\cdot p) \to X$ of degree $m_j = \deg(M_j) = d_{j+1} - d_{j} + \sg_H(k)$, we may write, for simplicity 
$\varphi_{j} \in H^{0}(M_j)$ where $\varphi_{1}^k=\phi_{21}^k$ and $\varphi_{2}^k=\phi_{32}^k$, and hence 
$M_j = \sO(D_j)$ where $D_j = \divv(\varphi_j)$. 

\noindent Note that $\varphi_j \neq 0 \Rightarrow m_j \geq 0$. Furthermore
\[
d_3 = d - d_1 - d_2 \Longleftrightarrow d_3 = \frac{d + 2m_2 + m_1 - 3\sg_H(k)}{3}
\]
and hence $d + m_1 + 2m_2 = 0\mod3$.

\noindent Using all the above notation, we can re-write the $(1,1,1)$-critical submanifolds as
\[
F^k_{\lambda} =
F_{m_1 m_2}^{k} =
\Bigg\{
\Big(
L_1\oplus L_2\oplus L_3,
\left(
 \begin{array}{c c c}
     0     &     0     & 0\\ 
 \varphi^k_{1} &     0     & 0\\
     0     & \varphi^k_{2} & 0
 \end{array}
\right)
\Big)
\Bigg\}\sse
\sM^{k}(3,d),
\]
and conclude that
\begin{Prop}[{\cite[Proposi\c{c}\~ao~2.3.9.]{ben}}]\label{bento2.3.9.}
For each pair $(m_1,m_2)\in \Omega$, there is a $(1,1,1)$-critical submanifold $F^{k}_{m_1 m_2}\sse \sM^k(3,d)$,  where
\[
\Omega = 
\Bigg\{
(x,y)\in \bN^{*}\times \bN^{*}\:
\begin{array}{c}
     d + x + 2y = 0\mod3\\
     2x + y < 3\sg_H(k)\\
     x + 2y < 3\sg_H(k)
\end{array}
\Bigg\}.
\]
\end{Prop}

\begin{proof}
 The stability conditions in this case are
 \[
 \mu(L_2 \oplus L_3) < \mu(E) 
 \Longleftrightarrow
 \frac{d_2 + d_3}{2} < \frac{d}{3}
 \word{and}
 \mu(L_3) < \mu(E)
 \Longleftrightarrow
 d_3 < \frac{d}{3}.
 \]
 In terms of $m_j \geq 0$ we get
 \[
  d_3 < \frac{d}{3} \Longleftrightarrow
  2m_2 +  m_1 < 3\sg_H(k)
 \word{and}
 \frac{d_2 + d_3}{2} < \frac{d}{3}
 \Longleftrightarrow
 2m_1 + m_2 < 3\sg_H(k)
 \]
\end{proof}

\begin{Rmk}
 For the particular case of $(1,1,1)$-critical submanifolds, note that 
 \[
  (\varphi_j^k \ox s_p)(L_j) \sse 
  (\varphi_j^k)(L_j) \ox L_p =
  L_{j+1} \ox K \ox L_p^{\ox k} \ox L_p =
  L_{j+1} \ox K \ox L_p^{\ox k+1}
 \]
and hence
\[
 i_k(F_{m_1 m_2}^{k}) \subseteq F_{m_1 m_2}^{k+1}.
\]
\end{Rmk}

\begin{Th}\label{(1,1,1)-VHS--SymP-Iso}
There is an isomorphism 
\[
F_{m_1 m_2}^{k} \cong 
\Sym^{\bar{m}_1+k}(X) \times \Sym^{\bar{m}_2+k}(X) \times \sJ^{d_3}(X)
\]
for each pair $(m_1,m_2)\in \Omega$, where 
$J^{d_3}(X)$ is the Jacobian of $X$, the moduli space of stable line bundles of degree $d_3$,
and $\bar{m}_j = m_j - k = d_{j+1} - d_j + 2g - 2$.
\end{Th}

\begin{proof}
 It is enough to take
 \[
  \begin{xy}
    (0,0)*+{F^{k}_{m_1 m_2}}="a";
    (65,0)*+{\Sym^{\bar{m}_1+k}(X) \times \Sym^{\bar{m}_2+k}(X) \times \sJ^{d_3}(X)}="b";
    {\ar@{->} "a";"b"};
    (0,-15)*+{(E,\varphi^k) = 
    (
 L_1\oplus L_2\oplus L_3,
 \left(
 \begin{array}{c c c}
     0       &     0       & 0\\ 
 \varphi^k_{1} &     0       & 0\\
     0       & \varphi^k_{2} & 0
 \end{array}
 \right)
 )
    }="c";
    (80,-15)*+{(\divv(\varphi^k_1),\divv(\varphi^k_2),L_3)}="d";
    {\ar@{|->} "c";"d"}.
  \end{xy}
\]
\end{proof}

\noindent In this case, the Morse index for $(1,1,1)$-critical submanifolds $F^k_{\lambda} = F^k_{m_1m_2}$ is given by 
$\lambda = 2\big(4(2g - 2)- m _1 - m_2 + 3k\big)$. The reader may consult Gothen~\cite{got} or Bento~\cite{ben} 
for details.

\begin{Rmk}
 Lemma~\ref{bento2.3.1.}, Proposition~\ref{bento2.3.2.}, Lemma~\ref{bento2.3.5.}, Proposition~\ref{bento2.3.6.}, 
 and Proposition~\ref{bento2.3.9.} are presented by Bento~\cite{ben} for the general case of rank three Hitchin 
 pairs. Here, we presented them for the particular case of rank three $k$-Higgs bundles.
\end{Rmk}

\noindent From the embeddings 
\[
F^k_{\la}  
\xrightarrow{\quad i_k \quad} 
F^{k+1}_{\la}\quad \forall \la,
\]
above mentioned, we get induced isomorphisms in cohomology:
\[
H^{j}(F_{\la}^{k+1},\mathbb{Z}) 
\xrightarrow{\quad \cong \quad} 
H^{j}(F_{\la}^{k},\mathbb{Z})
\]
for all $\la$, for certain values of $j$ in terms of $k$. Our goal is to find the range of~$j$ for which these isomorphisms hold.

\noindent The embeddings restricted to $(1,1)$-critical submanifolds in the rank two case, were studied by Hausel~\cite{hau} and presented by 
Hausel and Thaddeus~\cite{hath2}. Here, we focus on rank three.

\noindent If we restrict the embeddings to critical manifolds of type $(1,2)$:
\begin{equation}
  \begin{array}{r c l}
 F^k_{d_1} & \xrightarrow{\quad i_k \quad} & F^{k+1}_{d_1}\\
 & & \\
 \biggl(E_1 \oplus E_2, 
\left(
  \begin{array}{c c}
   0 & 0\\
   \varphi^{k}_{21} & 0
  \end{array}
\right)
 \biggl) &
\longmapsto &
 \biggl(E_1 \oplus E_2, 
\left(
  \begin{array}{c c}
   0 & 0\\
   \varphi^{k}_{21} \ox s_p & 0
  \end{array}
\right)
 \biggl)
  \end{array}\label{eq:(1,2)-VHS-embedding} 
\end{equation}
then, the isomorphisms
\[
\begin{array}{r c l}
 F^k_{d_1} & \xrightarrow{\quad \cong \quad} & \sN_{\sg_H(k)}(2,1,\td_1,\td_2)\\
 & & \\
 \biggl(E_1 \oplus E_2, 
\left(
  \begin{array}{c c}
   0 & 0\\
   \varphi^{k}_{21} & 0
  \end{array}
\right)
 \biggl) &
\longmapsto &
(V_1,V_2,\varphi)
\end{array}
\]
between $(1,2)$ critical submanifolds and moduli spaces of triples, where we denote 
by $V_1 = E_2 \ox K(kp)$, 
by $V_2 = E_1$, 
by $\varphi = \varphi^{k}_{21}$  
and $\sg_H(k) = \deg(K(kp)) = 2g - 2 + k$, 
induce another embeddings:
\[
i_k: 
\sN_{\sg_H(k)}(2,1,\td_1,\td_2) 
\to
\sN_{\sg_H(k+1)}(2,1,\td_1+2,\td_2)
\]
\[
(V_1,V_2,\varphi) 
\mapsto 
(V_1 \ox L_p,V_2,\varphi \ox s_p)
\]
where $\td_1 = \deg(V_1) = d_2 + 2\sg_H(k)$ 
and $\td_2 = \deg(V_2) = d_1$, and so,
induce embeddings on the flips:
\[
i_k: \sN_{\sg_H^{-}(k)}(2,1,\td_1,\td_2) 
\hookrightarrow
\sN_{\sg_H^{-}(k+1)}(2,1,\td_1+2,\td_2)
\]
and
\[
i_k: \sN_{\sg_H^{+}(k)}(2,1,\td_1,\td_2)
\hookrightarrow
\sN_{\sg_H^{+}(k+1)}(2,1,\td_1+2,\td_2).
\]

\noindent The situation with critical manifolds of type $(2,1)$
\begin{equation}
  \begin{array}{r c l}
 F^k_{d_2} & \xrightarrow{\quad i_k \quad} & F^{k+1}_{d_2}\\
 & & \\
 \biggl(E_2 \oplus E_1, 
\left(
  \begin{array}{c c}
   0 & 0\\
   \varphi^{k}_{21} & 0
  \end{array}
\right)
 \biggl) &
\longmapsto &
 \biggl(E_2 \oplus E_1, 
\left(
  \begin{array}{c c}
   0 & 0\\
   \varphi^{k}_{21} \ox s_p & 0
  \end{array}
\right)
 \biggl)
  \end{array}\label{eq:(2,1)-VHS-embedding} 
\end{equation}
is very similar to the $(1,2)$ situation, using now isomorphisms
\[
\begin{array}{r c l}
 F^k_{d_2} & \xrightarrow{\quad \cong \quad} & \sN_{\sg_H(k)}(1,2,\td_1,\td_2)\\
 & & \\
 \biggl(E_2 \oplus E_1, 
\left(
  \begin{array}{c c}
   0 & 0\\
   \varphi^{k}_{21} & 0
  \end{array}
\right)
 \biggl) &
\longmapsto &
(V_1,V_2,\varphi)
\end{array}
\]
where
by $V_1 = E_1 \ox K(k\cdot p)$, 
by $V_2 = E_2$, 
by $\varphi = \varphi^{k}_{21}$  
and $\sg_H(k) = \deg\big(K(k\cdot p)\big) = 2g - 2 + k$, 
and the induced embeddings become:
\[
i_k: 
\sN_{\sg_H(k)}(1,2,\td_1,\td_2) 
\to
\sN_{\sg_H(k+1)}(1,2,\td_1 + 1,\td_2)
\]
\[
(V_1,V_2,\varphi) 
\mapsto 
(V_1 \ox L_p,V_2,\varphi \ox s_p)
\]
where $\td_1 = \deg(V_1) = d_1 + \sg_H(k)$ 
and $\td_2 = \deg(V_2) = d_2$. Hence, for the flips on $\sg_H(k)$, the induced embeddings become:
\[
i_k\: \sN_{\sg_H^{-}(k)}(1,2,\td_1,\td_2) 
\hookrightarrow
\sN_{\sg_H^{-}(k+1)}(1,2,\td_1 + 1,\td_2)
\]
and
\[
i_k: \sN_{\sg_H^{+}(k)}(1,2,\td_1,\td_2)
\hookrightarrow
\sN_{\sg_H^{+}(k+1)}(1,2,\td_1 + 1,\td_2).
\]

\noindent Critical submanifolds of type $(1,1,1)$ are different from the other two. The embeddings
\begin{equation}
  \begin{array}{r c l}
 F^k_{m_1 m_2} & \xrightarrow{\quad i_k \quad} & F^{k+1}_{m_1 m_2}\\
 & & \\
 (
 L_1\oplus L_2\oplus L_3,
 \left(
 \begin{array}{c c c}
     0       &     0       & 0\\ 
 \varphi_{1} &     0       & 0\\
     0       & \varphi_{2} & 0
 \end{array}
 \right)
 ) &
\longmapsto &
 (
 L_1\oplus L_2\oplus L_3,
 \left(
 \begin{array}{c c c}
     0       &     0       & 0\\ 
 \varphi_{1}\ox s_p &     0       & 0\\
     0       & \varphi_{2}\ox s_p & 0
 \end{array}
 \right)
 )
  \end{array}\label{eq:(1,1,1)-VHS-embedding} 
\end{equation}
together with the isomorphisms
\[
F_{m_1 m_2}^{k} \cong 
\Sym^{\bar{m}_1+k}(X) \times \Sym^{\bar{m}_2+k}(X) \times \sJ^{d_3}(X)
\]
induce embeddings of the form:
{\footnotesize
\[
\begin{array}{r c l}
 \Sym^{\bar{m}_1 + k}(X)\times \Sym^{\bar{m}_2 + k}(X)\times \sJ^{d_3}(X) & \to & \Sym^{\bar{m}_1 + k + 1}(X)\times \Sym^{\bar{m}_2 + k + 1}(X)\times \sJ^{d_3}(X)\\
 \big(\divv(\varphi^k_1),\divv(\varphi^k_2),L_3\big) & \mapsto & \big(\divv(\varphi^k_1 + p),\divv(\varphi^k_2 + p),L_3\big).
\end{array}
\]
}

\section[Triples and Roof Theorem]{Stable Holomorphic Triples and Roof Theorem}
\label{sec:2}

\subsection[sigma-Stability]{$\sg$-Stability}
\label{ssec:2.1}

For $(2,1,\td_2,\td_1)$-triples and $(1,2,\td_1,\td_2)$-triples, the embeddings $i_k$ preserve $\sg$-stability:

\begin{Lem}\label{[GoZR3]}
 A triple $T$ of type $(2,1,\td_2,\td_1)$ or type $(1,2,\td_1,\td_2)$ is $\sg$-stable $\Leftrightarrow i_k(T)$ is $(\sg + 1)$-stable.
\end{Lem}

\begin{proof}
We will show the result holds for $(2,1,\td_2,\td_1)$-triples, the proof of $(1,2,\td_1,\td_2)$-triples is analogous.

\noindent Recall that $T = (V_1,V_2,\varphi)$ is $\sg$-stable if and only if $\mu_{\sg}(T') < \mu_{\sg}(T)$ for any $T'$ proper subtriple of $T$.

\noindent Denote by $S = i_k(T) = (V_1 \ox L_p,V_2,\varphi \ox s_p)$. Is easy to check that $\mu_{\sg + 1}(S) = \mu_{\sg}(T) + 1$:
\[
\mu_{\sg+1}(S) =
\frac{\deg_{\sg + 1}(S)}{\rk(V_1\ox L_p) \oplus \rk(V_2)} =
\]
\[
\frac{\deg(V_1\ox L_p) + \deg(V_2) + (\sg + 1)\rk(V_2)}{1 + 2} =
\]
\[
\frac{\deg(V_1) + \deg(L_p) + \deg(V_2) + \sg \rk(V_2) + \rk(V_2)}{3} =
\]
\[
\frac{\deg(V_1) + \deg(V_2) + \sg \rk(V_2)}{3} +  
\frac{\deg(L_p) + \rk(V_2)}{3} =
\mu_{\sg}(T) + 1
\]
since $\deg(L_p) = 1$ and $\rk(V_2) = 2$.

\noindent Any $S'$ proper subtriple of $S$ is of the form $S' = i_k(T')$ for some $T'$ subtriple of $T$, or equivalently:
\[
S' = (V'_1 \ox L_p,V'_2,\varphi \ox s_p)
\]
and there are injective sheaf homomorphisms $V'_1 \rightarrow V_1$ and $V'_2 \rightarrow V_2$. This 
statement is justified since the following diagram commutes:
\[
  \begin{xy}
    (0,0)*+{S}="a";
    (20,0)*+{S'}="b";
    (40,0)*+{T'}="c";
    (60,0)*+{T}="d";
    (0,-10)*+{V_2}="e";
    (20,-10)*+{B}="f";
    (40,-10)*+{B}="g";
    (60,-10)*+{V_2}="h";
    (-10,-30)*+{V_1 \ox L_p}="i";
    (20,-30)*+{A}="j";
    (50,-30)*+{A \ox L_{p}^{*}}="k";
    (80,-30)*+{V_1}="l";
    {\ar@{->>}^{\supsetneq} "a";"b"};
    {\ar@{^{(}->}^{\subsetneq} "c";"d"};
    {\ar@{->} "b";"c"};
    {\ar@{|->}^{\supsetneq} "e";"f"};
    {\ar@{|->}^{\subsetneq} "g";"h"};
    {\ar@{|->}^{=} "f";"g"};
    {\ar@{|->}^{\supsetneq} "i";"j"};
    {\ar@{|->}^{\subsetneq} "k";"l"};
    {\ar@{|->} "j";"k"};
    {\ar@{->}_{\varphi \ox s_p} "e";"i"};
    {\ar@{->}_{\varphi \ox s_p} "f";"j"};
    {\ar@{->}^{(\varphi \ox s_p) \ox s_p^{-1}} "g";"k"};
    {\ar@{->}^{\varphi} "h";"l"};
  \end{xy}
  \]  
where the first floor of the diagram contains the first entries of the triples, second floor contains the second entries, 
the diagonal arrows are the coresponding morphisms, and we consider the subbundles $A = V'_1 \ox L_p,\ B = V'_2$ and 
$T' = (V'_1,V'_2,\varphi) \subseteq (V_1,V_2,\varphi) = T$. So, there is a one--to--one correspondence between the proper 
subtriples $S' \subseteq S$ and the proper subtriples $T' \subseteq T$. We can easily see that $\mu_{\sg+1}(S') = \mu_{\sg}(T')+1$ and hence:
\[
\mu_{\sg+1}(S') < \mu_{\sg+1}(S) 
\Leftrightarrow 
\mu_{\sg}(T')+1 < \mu_{\sg}(T)+1 
\Leftrightarrow 
\mu_{\sg}(T') < \mu_{\sg}(T).
\]
Therefore, $T$ is $\sg$-stable $\Leftrightarrow S = i_k(T)$ is $(\sg + 1)$-stable. 
\end{proof}

\begin{Cor}\label{(2,1)-embedding}
 The embedding
 $$
 i_k: \sN_{\sg(k)}(2,1,\td_1,\td_2)
 \to 
 \sN_{\sg(k+1)}(2,1,\td_1+2,\td_2)
 $$
is well defined for any $\sg(k)$ such that $
\sg_m < \sg(k) < \sg_M$. In particular, 
the embedding $i_k$ restricted to $F_{d_1}^{k}$ 
(see (\ref{eq:(1,2)-VHS-embedding})) is well defined and 
we have a commutative diagram of the form:
 \begin{align*}
  \begin{xy}
    (0,0)*+{(\tilde{E}_1,\tilde{E}_2,\varphi_{21}^{k})}="a0";
    (0,10)*+{\sN_{\sg_H(k)}}="a1";
    (0,40)*+{F_{d_1}^{k}}="a2";
    (0,50)*+{(E_1 \oplus E_2, \varphi^k)}="a3";
    (50,0)*+{(\tilde{E}_1,\tilde{E}_2,\varphi_{21}^{k}\ox s_p),}="b0";
    (50,10)*+{\sN_{\sg_H(k+1)}}="b1";
    (50,40)*+{F_{d_1}^{k+1}}="b2";
    (50,50)*+{(E_1 \oplus E_2, \varphi^k \ox s_p)}="b3";
    {\ar@{<->}^{\cong} "a1";"a2" **\dir{--}};
    {\ar@{<->}_{\cong} "b1";"b2" **\dir{--}};
    {\ar@{->}_{i_k} "a1";"b1" **\dir{--}};
    {\ar@{->}^{i_k} "a2";"b2" **\dir{--}};
    {\ar@{|->} "a0";"b0" **\dir{--}};
    {\ar@{|->} "a3";"b3" **\dir{--}};
  \end{xy}
 \end{align*}
 where $\tilde{E}_1 = E_2 \ox K(k\cdot p)$, $\tilde{E}_2 = E_1$, and 
 $\varphi_{21}^{k}: E_1 \to E_2 \ox K(k\cdot p)$. \QEDA
\end{Cor}

\begin{Cor}\label{(1,2)-embedding}
 The embedding
 \[
 i_k: \sN_{\sg(k)}(1,2,\td_2,\td_1)
 \to 
 \sN_{\sg(k+1)}(1,2,\td_2 + 1,\td_1)
 \]
is well defined for any $\sg(k)$ such that $
\sg_m < \sg(k) < \sg_M$. In particular, 
the embedding $i_k$ restricted to $F_{d_2}^{k}$ 
(see (\ref{eq:(2,1)-VHS-embedding})) is well defined and 
we have a commutative diagram of the form:
 \begin{align*}
  \begin{xy} 
    (0,0)*+{(\tilde{E}_1,\tilde{E}_2,\varphi_{21}^{k})}="a0";
    (0,10)*+{\sN_{\sg_H(k)}}="a1";
    (0,40)*+{F_{d_2}^{k}}="a2";
    (0,50)*+{(E_2 \oplus E_1, \varphi^k)}="a3";
    (50,0)*+{(\tilde{E}_1,\tilde{E}_2,\varphi_{21}^{k}\ox s_p),}="b0";
    (50,10)*+{\sN_{\sg_H(k+1)}}="b1";
    (50,40)*+{F_{d_2}^{k+1}}="b2";
    (50,50)*+{(E_2 \oplus E_1, \varphi^k \ox s_p)}="b3";
    {\ar@{<->}^{\cong} "a1";"a2" **\dir{--}};
    {\ar@{<->}_{\cong} "b1";"b2" **\dir{--}};
    {\ar@{->}_{i_k} "a1";"b1" **\dir{--}};
    {\ar@{->}^{i_k} "a2";"b2" **\dir{--}};
    {\ar@{|->} "a0";"b0" **\dir{--}};
    {\ar@{|->} "a3";"b3" **\dir{--}};
  \end{xy}
 \end{align*}
 where $\tilde{E}_1 = E_1 \ox K(k\cdot p)$, $\tilde{E}_2 = E_2$, and 
 $\varphi_{21}^{k}: E_2 \to E_1 \ox K(k\cdot p)$. \QEDA
\end{Cor}

\noindent These results allow us to conclude that there is an interesting and important correspondence between the $\sg$-stability 
values of moduli spaces of holomorphic triples:
\begin{align*}
 \begin{xy}
  (0,35)*+{\sg_m(k)}="a1";
  (40,35)*+{\sg_H(k)}="b1";
  (80,35)*+{\sg_M(k)}="e1";
  (0,30)*+{|}="a2";
  (40,30)*+{*}="b2";
  (65,30)*+{|}="c2";
  (70,30)*+{\cdot}="d2";
  (75,30)*+{\cdot}="e2";
  (80,30)*+{|}="f2";
  (100,30)*+{}="g2";
  (0,5)*+{\cdot}="a3";
  (5,5)*+{|}="b3";
  (40,5)*+{\cdot}="c3";
  (45,5)*+{*}="d3";
  (70,5)*+{|}="e3";
  (75,5)*+{\cdot}="f3";
  (80,5)*+{\cdot}="g3";
  (85,5)*+{|}="h3";
  (90,5)*+{\cdot}="i3";
  (95,5)*+{\cdot}="j3";
  (100,5)*+{|}="k3";
  (120,5)*+{}="l3";
  (5,0)*+{\sg_m(k+1)}="a4";
  (45,0)*+{\sg_H(k+1)}="b4";
  (85,0)*+{\sg'}="c4";
  (100,0)*+{\sg_M(k+1)}="e4";
  {\ar@{-} "a2";"b2"};
  {\ar@{-} "b2";"c2"};
  {\ar@{-} "c2";"d2"};
  {\ar@{-} "d2";"e2"};
  {\ar@{-} "e2";"f2"};
  {\ar@{->} "f2";"g2"};
  {\ar@{->}^{i_k} "a2";"b3"};
  {\ar@{->}^{i_k} "b2";"d3"};
  {\ar@{->}^{i_k} "f2";"h3"};
  {\ar@{->}^{i_k} "c2";"e3"};
  {\ar@{-->} "f2";"k3"};
  {\ar@{-} "a3";"b3"};
  {\ar@{-} "b3";"c3"};
  {\ar@{-} "c3";"d3"};
  {\ar@{-} "d3";"e3"};
  {\ar@{-} "e3";"f3"};
  {\ar@{-} "f3";"g3"};
  {\ar@{-} "g3";"h3"};
  {\ar@{-} "h3";"i3"};
  {\ar@{-} "i3";"j3"};
  {\ar@{-} "j3";"k3"};
  {\ar@{->} "k3";"l3"};
  \end{xy}
\end{align*}
where $\sg_m(k) = \tilde{\mu}_1 - \tilde{\mu}_2$, 
$\sg_M(k) = 4(\tilde{\mu}_1 - \tilde{\mu}_2)$, 
$\sg_H(k) = \deg(K(kp)) = 2g - 2 + k$, 
and the correspondence gives us
$\sg_m(k+1) = \sg_m(k) + 1$, 
$\sg' = \sg_M(k) + 1$, 
$\sg_M(k+1) = \sg_M(k)+3$, and 
$\sg_H(k+1) = \sg_H(k) + 1$. First and second floor are representations of the real line, where the second floor of the diagram corresponds to the interval $[\sg_m,\sg_M]$ for poles of order $k$, and the first floor for poles of order $(k + 1)$ after the embedding $i_k$.

\begin{Rmk}
 An interesting fact from the correspondence represented by last diagram is that 
 \[
  i_k\:
  \sg_M(k) \mapsto \sg' < \sg_M(k+1).
 \]
\end{Rmk}

\subsection[Blow-up and The Roof Theorem]{Blow-up and The Roof Theorem}
\label{ssec:2.2}

At this point, a brief description of the flip loci of the moduli spaces of holomorphic triples will be useful to understand the coming results and notation. The reader may see Mu\~noz~et~al.~\cite{mov} for details.

\noindent Fixing the type $(r_1,r_2,d_1,d_2)$ for the moduli spaces of holomorphic triples, we shall describe the differences between $\sN_{\sg_1}(r_1,r_2,d_1,d_2)$ and $\sN_{\sg_2}(r_1,r_2,d_1,d_2)$ where $\sg_1$ and $\sg_2$ are separated by a critical value $\sg_c \in [\sg_m,\sg_M]$. Here, we suppose $r_1 \neq r_2$, since for our purposes, the case $r_1 = 2$ and $r_2 = 1$ will be particularly useful. 

\noindent Let 
\[
 \sg_{c}^{+} = \sg_c + \varepsilon
 \word{and}
 \sg_{c}^{-} = \sg_c - \varepsilon
\]
where $\varepsilon > 0$ is small enough so that $\sg_c \in \ ]\sg_{c}^{-},\sg_{c}^{+}[$ is the only critical value in that subinterval.

\begin{Def}\label{fliploci}
 Define the {\em flip loci} as the sets
 \[
  S_{\sg_{c}^{+}} =
  \left\{ 
  T\in \sN_{\sg_{c}^{+}} |\ T \word{is} \sg_{c}^{-}-\textmd{unstable}
  \right\}
  \sse \sN_{\sg_{c}^{+}}(r_1,r_2,d_1,d_2)
 \]
and
 \[
  S_{\sg_{c}^{-}} =
  \left\{ 
  T\in \sN_{\sg_{c}^{-}} |\ T \word{is} \sg_{c}^{+}-\textmd{unstable}
  \right\}
  \sse \sN_{\sg_{c}^{-}}(r_1,r_2,d_1,d_2),
 \]
and denote 
$
S_{\sg_{c}^{\pm}}^{s} =
S_{\sg_{c}^{\pm}} \cap 
\sN_{\sg_{c}^{\pm}}^{s}(r_1,r_2,d_1,d_2)
$
as the {\em stable part of the flip loci}, where $\sg_{c}^{\pm}$ means any of both $\sg_{c}^{+}$ or $\sg_{c}^{-}$.
\end{Def}

\noindent Denote $\tilde{\sN}_{\sg_c^{-}(k)}$ as the blow-up of $\sN_{\sg_c^{-}(k)} = \sN_{\sg_c^{-}(k)}(2,1,\td_1,\td_2)$ 
along the flip locus $S_{\sg_c^{-}(k)}$, which is isomorphic to $\tilde{\sN}_{\sg_c^{+}(k)}$, the blow-up of 
$\sN_{\sg_c^{+}(k)} = \sN_{\sg_c^{+}(k)}(2,1,\td_1,\td_2)$ along the flip locus $S_{\sg_c^{+}(k)}$. From now on, we 
will denote just $\tilde{\sN}_{\sg_c(k)}$ whenever no confusion is likely to arise. 

\begin{Th}\label{RoofTheorem}
 For each $k$, there exists an embedding at the blow-up level 
 \[
  \tilde{i_k}: \tilde{\sN}_{\sg_c(k)} \hookrightarrow \tilde{\sN}_{\sg_c(k+1)}
 \]
 such that the following diagram commutes: 
 \begin{align*}
 \begin{xy}
(0,-20)*+{\sN_{\sg_c^{-}(k+1)}}="a";
(20,0)*+{\tilde{\sN}_{\sg_c(k+1)}}="b";
(10,-50)*+{\tilde{\sN}_{\sg_c(k)}}="c";
(40,-20)*+{\sN_{\sg_c^{+}(k+1)}}="d";
(-10,-70)*+{\sN_{\sg_c^{-}(k)}}="e";
(30,-70)*+{\sN_{\sg_c^{+}(k)}}="f";
{\ar@{-->}^(.45){\exists \tilde{i_k}} "c";"b" **\dir{--}};
{\ar^(.45){} "b";"a" **\dir{-}};
{\ar^(.45){i_k} "e";"a" **\dir{-}};
{\ar^(.45){} "c";"e" **\dir{-}};
{\ar_(.45){i_k} "f";"d" **\dir{-}};
{\ar^(.45){} "c";"f" **\dir{-}};
{\ar^(.45){} "b";"d" **\dir{-}};
 \end{xy}
 \end{align*}
where $\tilde{\sN}_{\sg_c(k)}$ is the blow-up of 
$\sN_{\sg_c^{-}(k)} = \sN_{\sg_c^{-}(k)}(2,1,\td_1,\td_2)$ along the flip locus $S_{\sg_c^{-}(k)}$ 
and, at the same time, represents the blow-up of
$\sN_{\sg_c^{+}(k)} = \sN_{\sg_c^{+}(k)}(2,1,\td_1,\td_2)$ along the flip locus $S_{\sg_c^{+}(k)}$.
\end{Th}

\begin{proof}
Recall that $T$ is $\sg$-stable if and only if $i_k(T)$ is $(\sg+1)$-stable. Furthermore, by~\cite{mov}, note that any triple 
\[
T = (V_1,V_2,\varphi)\in S_{\sg_c^{+}(k)} \subseteq \sN_{\sg_c^{+}(k)}(2,1,\td_1,\td_2) 
\]
is a non-trivial extension of a subtriple $T' \subseteq T$ of the form $T' = (V'_1,V'_2,\varphi') = (M,0,\varphi')$ by a quotient triple of 
the form $T'' = (V''_1,V''_2,\varphi'') = (L,V_2,\varphi'')$, where $M$ is a line bundle of degree $\deg(M) = d_M$
and $L$ is a line bundle of degree $\deg(L) = d_L = \td_1 - d_M$. Besides, also by ~\cite{mov}, the non-trivial critical 
values $\sg_c \neq \sg_m$ for $\sg_m < \sg_c < \sg_M$ are of the form $\sg_c = 3d_M - \td_1 - \td_2$. Then, 
we can visualize the embedding $i_k: T \hookto i_k(T)$ as follows:
\begin{align*}
 \begin{xy}
  (0,0)*+{0}="a1";
  (20,0)*+{T'}="b1";
  (40,0)*+{T}="c1";
  (60,0)*+{T''}="d1";
  (80,0)*+{0}="e1";
  (0,-10)*+{0}="a2";
  (20,-10)*+{0}="b2";
  (40,-10)*+{V_2}="c2";
  (60,-10)*+{V_2}="d2";
  (80,-10)*+{0}="e2";
  (-20,-30)*+{0}="a3";
  (0,-30)*+{M}="b3";
  (40,-30)*+{V_1}="c3";
  (80,-30)*+{L}="d3";
  (100,-30)*+{0}="e3";
  (0,-50)*+{0}="a4";
  (20,-50)*+{0}="b4";
  (40,-50)*+{V_2}="c4";
  (60,-50)*+{V_2}="d4";
  (80,-50)*+{0}="e4";
  (-20,-70)*+{0}="a5";
  (0,-70)*+{M\ox L_p}="b5";
  (40,-70)*+{V_1\ox L_p}="c5";
  (80,-70)*+{L\ox L_p}="d5";
  (100,-70)*+{0}="e5";
  {\ar@{->} "a1";"b1"};
  {\ar@{->} "b1";"c1"};
  {\ar@{->} "c1";"d1"};
  {\ar@{->} "d1";"e1"};
  {\ar@{->} "a2";"b2"};
  {\ar@{->} "b2";"c2"};
  {\ar@{->}^{=} "c2";"d2"};
  {\ar@{->} "d2";"e2"};
  {\ar@{->}_{\varphi'} "b2";"b3"};
  {\ar@{->}_{\varphi} "c2";"c3"};
  {\ar@{->}_{\varphi''} "d2";"d3"};
  {\ar@{->} "a3";"b3"};
  {\ar@{->} "b3";"c3"};
  {\ar@{->} "c3";"d3"};
  {\ar@{->} "d3";"e3"};
  {\ar@{|->}^{i_k} "c3";"c4"};
  {\ar@{->} "a4";"b4"};
  {\ar@{->} "b4";"c4"};
  {\ar@{->}^{=} "c4";"d4"};
  {\ar@{->} "d4";"e4"};
  {\ar@{->}_{\varphi' \ox s_p} "b4";"b5"};
  {\ar@{->}_{\varphi \ox s_p} "c4";"c5"};
  {\ar@{->}_{\varphi'' \ox s_p} "d4";"d5"};
  {\ar@{->} "a5";"b5"};
  {\ar@{->} "b5";"c5"};
  {\ar@{->} "c5";"d5"};
  {\ar@{->} "d5";"e5"};
 \end{xy}
\end{align*}

where $\deg(V_1 \ox L_p) = \td_1 + 2$ and $\deg(M \ox L_p) = d_M + 1$, and so $L \ox L_p$ 
verifies that $\deg(L \ox L_p) = \deg(V_1 \ox L_p) - \deg(M \ox L_p)$:
$$
\deg(L \ox L_p) = 
d_L + 1 = 
\td_1 - d_M + 1 = 
$$
$$
(\td_1 + 2) - (d_M + 1) = 
\deg(V_1 \ox L_p) - \deg(M \ox L_p).
$$
Hence, $\sg_c(k+1)$ verifies that $\sg_c(k+1) = \sg_c(k) + 1$:
$$\sg_c(k+1) = 3\deg(M \ox L_p) - \deg(V_1 \ox L_p) - \deg(V_2) = $$
$$3d_M + 3 - \td_1 - 2 - \td_2 = (3d_M - \td_1 - \td_2) + 1 = \sg_c(k) + 1$$
and where $i_k(T') = (M \ox L_p,0,\varphi' \ox s_p)$ is the maximal $\sg_c^{+}(k+1)$-destabilizing subtriple of $i_k(T)$,
verifying exactness at the image level of the embedding.

\noindent Similarly, also by~\cite{mov}, any triple $T \in S_{\sg_c^{-}(k)} \subseteq \sN_{\sg_c^{-}(k)}(2,1,\td_1,\td_2)$
is a non-trivial extension of a subtriple $T' \subseteq T$ of the form $T' = (V'_1,V'_2,\varphi') = (L,V_2,\varphi')$
by a quotient triple of the form $T'' = (V''_1,V''_2,\varphi'') = (M,0,\varphi'')$, where $M$ is a line bundle 
of degree $\deg(M) = d_M$ and $L$ is a line bundle of degree $\deg(L) = d_L = \td_1 - d_M$. 
Then, the embedding 
$$
i_k: T \hookto i_k(T)
$$ 
looks like:
\begin{align*}
 \begin{xy}
  (0,0)*+{0}="a1";
  (20,0)*+{T'}="b1";
  (40,0)*+{T}="c1";
  (60,0)*+{T''}="d1";
  (80,0)*+{0}="e1";
  (0,-10)*+{0}="a2";
  (20,-10)*+{V_2}="b2";
  (40,-10)*+{V_2}="c2";
  (60,-10)*+{0}="d2";
  (80,-10)*+{0}="e2";
  (-20,-30)*+{0}="a3";
  (0,-30)*+{L}="b3";
  (40,-30)*+{V_1}="c3";
  (80,-30)*+{M}="d3";
  (100,-30)*+{0}="e3";
  (0,-50)*+{0}="a4";
  (20,-50)*+{V_2}="b4";
  (40,-50)*+{V_2}="c4";
  (60,-50)*+{0}="d4";
  (80,-50)*+{0}="e4";
  (-20,-70)*+{0}="a5";
  (0,-70)*+{L\ox L_p}="b5";
  (40,-70)*+{V_1\ox L_p}="c5";
  (80,-70)*+{M\ox L_p}="d5";
  (100,-70)*+{0}="e5";
  {\ar@{->} "a1";"b1"};
  {\ar@{->} "b1";"c1"};
  {\ar@{->} "c1";"d1"};
  {\ar@{->} "d1";"e1"};
  {\ar@{->} "a2";"b2"};
  {\ar@{->}^{=} "b2";"c2"};
  {\ar@{->} "c2";"d2"};
  {\ar@{->} "d2";"e2"};
  {\ar@{->}_{\varphi'} "b2";"b3"};
  {\ar@{->}_{\varphi} "c2";"c3"};
  {\ar@{->}_{\varphi''} "d2";"d3"};
  {\ar@{->} "a3";"b3"};
  {\ar@{->} "b3";"c3"};
  {\ar@{->} "c3";"d3"};
  {\ar@{->} "d3";"e3"};
  {\ar@{|->}^{i_k} "c3";"c4"};
  {\ar@{->} "a4";"b4"};
  {\ar@{->}^{=} "b4";"c4"};
  {\ar@{->} "c4";"d4"};
  {\ar@{->} "d4";"e4"};
  {\ar@{->}_{\varphi' \ox s_p} "b4";"b5"};
  {\ar@{->}_{\varphi \ox s_p} "c4";"c5"};
  {\ar@{->}_{\varphi'' \ox s_p} "d4";"d5"};
  {\ar@{->} "a5";"b5"};
  {\ar@{->} "b5";"c5"};
  {\ar@{->} "c5";"d5"};
  {\ar@{->} "d5";"e5"};
 \end{xy}
\end{align*}
where $i_k(T') = (L,V_2,\varphi')$ is the maximal $\sg_c^{+}(k+1)$-destabilizing subtriple of $i_k(T)$.

\noindent Hence, $i_k$ restricts to the flip loci $S_{\sg_c^{+}(k)}$ and $S_{\sg_c^{-}(k)}$. Recall that, by definition, the blow-up of
$\sN_{\sg_c^{+}(k)}$ along the flip locus $S_{\sg_c^{+}(k)}$, is the space $\tilde{\sN}_{\sg_c(k)}$ together 
with the projection 
\[
\pi:\ \tilde{\sN}_{\sg_c(k)} \rightarrow \sN_{\sg_c^{+}(k)} 
\]
where $\pi$ restricted to $\sN_{\sg_c^{+}(k)} - S_{\sg_c^{+}(k)}$ is an isomorphism and the \emph{exceptional divisor} 
$\sE^{+} = \pi^{-1}(S_{\sg_c^{+}(k)}) \subseteq \tilde{\sN}_{\sg_c(k)}$ is a fiber bundle over $S_{\sg_c^{+}(k)}$ 
with fiber $\mathbb{P}^{n-k-1}$, where $n = \mathrm{dim}(\sN_{\sg_c^{+}(k)})$ and $k = \mathrm{dim}(S_{\sg_c^{+}(k)})$. So, 
the embedding can be extended to $\sE^{+}$ in a natural way. Same argument remains valid when we consider 
$\tilde{\sN}_{\sg_c(k)}$ as the blow-up of $\sN_{\sg_c^{-}(k)}$ along the flip locus $S_{\sg_c^{-}(k)}$ with 
exceptional divisor $\sE^{-} = \pi^{-1}(S_{\sg_c^{-}(k)}) \subseteq \tilde{\sN}_{\sg_c(k)}$. Therefore, the embedding 
can be extended to the whole $\tilde{\sN}_{\sg_c(k)}$. 
\end{proof}

\noindent Recall that there is an isomorphism
 \[
  \sN_{\sg}(1,2,d_1,d_2)
  \cong
  \sN_{\sg}(2,1,-d_2,-d_1)
 \]
for all $\sg$ by Proposition~\ref{dualtriples}. Hence, the following corollary represents the analogous dual Roof-Theorem for the $(1,2)$-case, and also holds:

\begin{Cor}\label{(1,2)-RoofTheorem}
 For each $k$, there exists an embedding at the blow-up level 
 \[
  \tilde{i_k}: \tilde{\sN}_{\sg_c(k)} \hookrightarrow \tilde{\sN}_{\sg_c(k+1)}
 \]
 such that the following diagram commutes: 
 \begin{align*}
 \begin{xy}
(0,-20)*+{\sN_{\sg_c^{-}(k+1)}}="a";
(20,0)*+{\tilde{\sN}_{\sg_c(k+1)}}="b";
(10,-50)*+{\tilde{\sN}_{\sg_c(k)}}="c";
(40,-20)*+{\sN_{\sg_c^{+}(k+1)}}="d";
(-10,-70)*+{\sN_{\sg_c^{-}(k)}}="e";
(30,-70)*+{\sN_{\sg_c^{+}(k)}}="f";
{\ar@{-->}^(.45){\exists \tilde{i_k}} "c";"b" **\dir{--}};
{\ar^(.45){} "b";"a" **\dir{-}};
{\ar^(.45){i_k} "e";"a" **\dir{-}};
{\ar^(.45){} "c";"e" **\dir{-}};
{\ar_(.45){i_k} "f";"d" **\dir{-}};
{\ar^(.45){} "c";"f" **\dir{-}};
{\ar^(.45){} "b";"d" **\dir{-}};
 \end{xy}
 \end{align*}
where $\tilde{\sN}_{\sg_c(k)}$ is the blow-up of 
$\sN_{\sg_c^{-}(k)} = \sN_{\sg_c^{-}(k)}(1,2,\td_2,\td_1)$ along the flip locus $S_{\sg_c^{-}(k)}$ 
and, at the same time, represents the blow-up of
$\sN_{\sg_c^{+}(k)} = \sN_{\sg_c^{+}(k)}(1,2,\td_2,\td_1)$ along the flip locus $S_{\sg_c^{+}(k)}$.
\end{Cor}

\begin{proof}
Follows from Theorem~\ref{RoofTheorem} and Proposition~\ref{dualtriples}.
\end{proof}

\begin{Rmk}
 The construction of the blow-up may be found in the book of Griffiths and Harris~\cite{grha}.
\end{Rmk}


\section[Cohomology]{Cohomology}
\label{sec:3}

We want to show that the embeddings
$
i_k\: F_{\lambda}^{k} \hookto F_{\lambda}^{k+1}
$ 
induce covariant isomorphisms in cohomology:
\[
H^{j}(F_{\lambda}^{k+1},\mathbb{Z}) 
\xrightarrow{\quad \cong \quad} 
H^{j}(F_{\lambda}^{k},\mathbb{Z}) 
\]
for all $\lambda$ and certain $j$. To do that, we need to study $F_{d_1}^k$, $F_{d_2}^k$ and $F_{m_1 m_2}^k$ separately. Because of 
Proposition \ref{dualtriples}, the cohomology of $F_{d_1}^k$ and $F_{d_2}^k$ are similar, so it will be enough to analyze $F_{d_1}^k$. The 
cohomology of $F_{m_1 m_2}^k$ will be completely different.
 
\noindent We shall start by describing the cohomology of $\Sym^{k}(X) = X^{k}/S_{k}$, the $k$-th symmetric product in subsection \ref{ssec:3.1}, which is related 
to the cohomology of the rank three VHS.

\noindent For $(1,2)$-VHS, we will prove that the embeddings
$
i_k\: F_{d_1}^{k} \hookto F_{d_1}^{k+1}
$ 
induce isomorphisms in cohomology:
\[
H^{j}(F_{d_1}^{k+1},\mathbb{Z}) 
\xrightarrow{\quad \cong \quad} 
H^{j}(F_{d_1}^{k},\mathbb{Z}) 
\]
for certain $j$, or equivalently:
\[
H^{j}(\sN_{\sg_H}^{k+1},\mathbb{Z}) 
\xrightarrow{\quad \cong \quad} 
H^{j}(\sN_{\sg_H}^{k},\mathbb{Z}), 
\]
where we denote 
$
\sN_{\sg_H}^{k} = 
\sN_{\sg_H(k)}(2,1,\td_1,\td_2)
$. We do that in two steps. First, in subsection \ref{ssec:3.2}, we get that 
\[
H^{j}(\sN_{\sg_c}^{k+1},\mathbb{Z}) 
\xrightarrow{\quad \cong \quad} 
H^{j}(\sN_{\sg_c}^{k},\mathbb{Z}) 
\]
for all critical $\sg_c = \sg_c(k)$ such that $\sg_m(k) < \sg_c(k) < \sg_M(k)$, and for 
all $j \leqslant n(k)$, where the bound $n(k)$ is known. We first analize the 
embedding restricted to the flip loci,
$
i_k: S_{\sg_c^-(k)} 
\hookrightarrow 
S_{\sg_c^-(k+1)}
$ 
and 
$
i_k: S_{\sg_c^+(k)} 
\hookrightarrow 
S_{\sg_c^+(k+1)}
$.
For simplicity, we will denote from now on 
$
S_{-}^{k} = 
S_{\sg_c^-(k)}
$ 
and 
$
S_{+}^{k} = 
S_{\sg_c^+(k)}
$ 
whenever no confusion is likely to arise about the critical value.

\noindent In subsection \ref{ssec:3.3}, we stabilize the cohomology of the $(1,2)$-VHS, using useful results from the work of {Bradlow,~Garc\'ia-Prada,~Gothen~\cite{bgg1}}. In subsection \ref{ssec:3.4}, we present the dual results for $(2,1)$-VHS. 

\noindent Finally, in subsection \ref{ssec:3.5}, we study the case of the $(1,1,1)$-VHS.

\subsection[Cohomology of Symmetric Products]{Cohomology of Symmetric Products}
\label{ssec:3.1}
%

\noindent We begin by recalling some cohomology features of $\Sym^{k}(X) = X^{k}/S_{k}$, the symmetric product with quotient topology, where $X^{k}$ is the 
$k$-times cartesian product and $S_{k}$ is the order $k$ symmetric group. Obviously $\Sym^{1}(X) = X$.

\noindent As mentioned before, the $k$-th symmetric product $\Sym^{k}(X)$ is a smooth projective variety of dimension $k\in \bN$, that could be 
interpretated as the moduli space of degree $k$ effective divisors.

\noindent It is well known that 
\[
 H^{0}(X,\bZ) = \bZ,\quad
 H^{1}(X,\bZ) = \bZ^{2g},\quad
 H^{2}(X,\bZ) = \bZ.
\]
There is a generator $\beta\in H^{2}(X,\bZ)$ induced by the orientation of $X$. Moreover, there are $2g$ generators 
$\al_1,\ \al_2,\ \dots,\ \al_{2g}\in H^{1}(X,\bZ)$ such that 
\[
 \al_{i} \cup \al_{j} = 
 -\al_{j} \cup \al_{i} = 
 0 \word{if} i - j \neq \pm g \word{for} i,j\in \{1,\ \dots, 2g\}
\]
and
\[
 \al_{i} \cup \al_{i + g} = 
 -\al_{i+g} \cup \al_{i} = 
 \beta \word{for} i\in \{1,\ \dots, g\}
\]
with the usual cup product $\cup$. Hence
\[
 \al_{i} \cup \beta = 
 \beta \cup \al_{i} = 0
 \word{and} 
 \beta^2 = \beta \cup \beta = 0.
\]

\noindent For the usual cartesian product $X^{k}$, we get that the ring $H^{*}(X^{k},\bZ) \cong H^{*}(X,\bZ)^{\ox k}$ is generated by
$\{\al_{ir}\}_{i=1}^{2g}$ and $\beta_{r}$ with $1\leq r \leq k$, which are elements of the form
\[
 \al_{ir} = 
 1\ox \dots \ox 1 \ox \al_{i} \ox 1 \ox \dots \ox 1 \in H^{1}(X^{k},\bZ)
\]
and
\[
 \beta_{r} = 
 1\ox \dots \ox 1 \ox \beta \ox 1 \ox \dots \ox 1 \in H^{2}(X^{k},\bZ)
\]
where $\al_{i}$ and $\beta$ fill the $r$-th entry of $\al_{ir}$ and $\beta_{r}$ respectively, and they are subject to the relations 
\[
 \al_{ir} \cup \al_{jr} = 
 -\al_{jr} \cup \al_{ir} = 
 0 \word{if} i - j \neq \pm g \word{for} i,j\in \{1,\ \dots, 2g\}
\]
and
\[
 \al_{ir} \cup \al_{i + g\ r} = 
 -\al_{i + g\ r} \cup \al_{ir} = 
 \beta_r \word{for} i\in \{1,\ \dots, g\}.
\]
Hence
\[
 \al_{ir} \cup \beta_{r} = \beta_{r} \cup \al_{ir} = 0
 \word{and} 
 \beta_{r}^2 = \beta_{r} \cup \beta_{r} = 0.
\]
Besides, each $\beta_{r}$ commutes with every element of $H^{*}(X^{k},\bZ)$.

\noindent Finally, the symmetric product $\Sym^{k}(X)$ has a cohomology ring $H^{*}(\Sym^{k}(X),\bZ)$ generated by elements of the form
\[
 \zeta_{i} =
 \al_{i1} + \dots + \al_{ik}
 = \sum_{r=1}^{k}\al_{ir}
 \in H^{1}(\Sym^{k}(X),\bZ)
 \word{for} 1\leq i \leq 2g
\]
and 
\[
 \eta =
 \beta_{1} + \dots + \beta_{k}
 = \sum_{r=1}^{k}\beta_{r}
 \in H^{2}(\Sym^{k}(X),\bZ)
\]
where
\[
\zeta_{i} \cup \zeta_{j}
=
-\zeta_{j} \cup \zeta_{i}
 \word{and}
 \zeta_{i} \cup \eta
=
\eta \cup \zeta_{i}
\]
for any $i$ and $j$. The reader may consult Macdonald~\cite{mac} or Arbarello-Cornalba-Griffiths-Harris~\cite{acgh} for details.

\noindent According to Arbarello~et al.~\cite{acgh}, there is $\bigtriangleup_{k}\in \Sym^{k+1}(X)$ 
a universal divisor such that 
\[
 \bigtriangleup_{k}\Big|_{\{D\}\times X} = D
 \word{for every divisor}
 D\in \Sym^{k}(X).
\]
Therefore, the first Chern class $c_1(\bigtriangleup_{k})\in H^{2}(\Sym^{k+1}(X),\bZ)$ of this universal divisor is given by
\begin{equation}\label{c1udiv}
 c_1(\bigtriangleup_{k})
 = \gamma \ox k +
 \sum_{i=1}^{g}(\zeta_{i} \ox \al_{i+g} - \zeta_{i + g} \ox \al_{i})
 + \eta \ox 1
 \in
 H^{2}(\Sym^{k+1}(X),\bZ)
\end{equation}
%
where 
\[
H^{2}(\Sym^{k+1}(X),\bZ) =
\sum_{j=0}^{2}H^{j}(\Sym^{k}(X),\bZ)\ox H^{2-j}(X,\bZ)
\]
and
\[
 \gamma =
 \sum_{i = 1}^{g}\zeta_{i}\cup \zeta_{i+g}
 \in H^{2}(\Sym^{k}(X),\bZ).
\]
Macdonald~\cite{mac} compute the Poincar\'e polynomial of $H^{*}(\Sym^{k}(X),\bZ)$:
\begin{equation}\label{poincaresymk}
 P_{t}\big(\Sym^{k}(X)\big) =
 \begin{array}{r}
  \word{Coeff}\\
  x^k
 \end{array}
 \left(
 \frac{(1 + xt)^{2g}}{(1 - x)(1 - xt^2)}
 \right).
\end{equation}

\noindent For $k > 2g - 2$ there is the Abel--Jacobi map $\Sym^k(X) \to \sJ^{k}$, which is a locally trivial fibration with fibre $\bP^{k - g}$, 
and gives the Poincar\'e polynomial:
\begin{equation}\label{poincaresymk2}
 P_{t}\big(\Sym^{k}(X)\big) =
 \left(
 \frac{(1 + t)^{2g}(1 + t^{2(k-g+1)})}{(1 - t^2)}
 \right).
\end{equation}

\noindent The reader may see Macdonald~\cite{mac}, Arbarello~et al.~\cite{acgh}, or Hausel~\cite{hau} for details.

\noindent Our embedding $i_k\: F_{\la}^{k}\to F_{\la}^{k+1}$ is in fact related to the embedding
\[
 \Sym^{k}(X)\to \Sym^{k+1}(X)
\]
\[
 D\mapsto D+p
\]
for a fixed point $p\in X$. We will abuse notation and call this last embedding also $i_k$. We get a sequence
\[
 X = 
 \Sym^{1}(X) \sse
 \Sym^{2}(X) \sse
 \dots \sse 
 \Sym^{k}(X) \sse
 \dots 
\]
and so, we may consider its direct limit 
\[
 \Sym^{\infty}(X) = 
 \lim_{k\to \infty}\Sym^{k}(X),
\]
which is a $\bP^{\infty}$-bundle over the Jacobian $\sJ$, and hence its Poincar\'e polynomial is:
\begin{equation}\label{poincaresyminfty}
 P_{t}\big(\Sym^{\infty}(X)\big) =
 \left(
 \frac{(1 + t)^{2g}}{(1 - t^2)}
 \right).
\end{equation}
The reader may consult Hausel~\cite{hau} for all the details.

\begin{Th}\label{symmetric-pullback}
 The pull-back
 \[
  i_{k}^{*}\:
  H^{*}(\Sym^{k+1}(X),\bZ)
  \to
  H^{*}(\Sym^{k}(X),\bZ)
 \]
induced by the embedding $i_{k}\: \Sym^{k}(X)\to \Sym^{k+1}(X)$, is surjective.
\end{Th}

\begin{proof}
 It is enough to see that the cohomology ring $H^{*}(\Sym^{k}(X),\bZ)$ is generated by the universal classes $\{\zeta_i\}_{i=1}^{g}$ and $\eta$ 
 mentioned above, and that the universal divisor $\bigtriangleup_{k}$ has first Chern class of the form~\ref{c1udiv}. See Hausel~\cite{hau} for details.
\end{proof}

\begin{Cor}\label{symmetric-directlimit}
 The cohomology ring of the direct limit $\Sym^{\infty}(X)$ is the covariant limit
 \[
  H^{*}(\Sym^{\infty}(X),\bZ)
  =
  \lim_{\infty \leftarrow k}
  H^{*}(\Sym^{k}(X),\bZ)
 \]
which is a graded commutative free algebra generated by the classes $\{\zeta_i\}_{i=1}^{g}$ and $\eta$. 
\end{Cor}

\begin{proof}
 This is a consequence of Theorem~\ref{symmetric-pullback} and the Poincar\'e polynomial~(\ref{poincaresyminfty}) found by Hausel~\cite{hau}.
\end{proof}

\begin{Th}[{\cite[(12.2)]{mac}}]\label{macdonald-12.2}
 There is a cohomology isomorphism
 \[
  H^{j}(\Sym^{k+1}(X),\bZ)
  \to
  H^{j}(\Sym^{k}(X),\bZ)
 \]
 for all $j \leq k-1$.\QEDA
\end{Th}

\begin{Cor}\label{symmetric-cohomology-iso}
 There is an isomorphism
 \[
  H^{j}(\Sym^{\infty}(X),\bZ)
  \to
  H^{j}(\Sym^{k}(X),\bZ)
 \]
 for all $j \leq k-1$.
\end{Cor}

\begin{proof}
 It follows directly from Theorem~\ref{symmetric-pullback}, Corollary~\ref{symmetric-directlimit} and Theorem~\ref{macdonald-12.2}.
\end{proof}

\subsection[Cohomology of Triples]{Cohomology of Triples}
\label{ssec:3.2}

A few words about notation. Recall that we are using $\td_j = \deg(V_j)$
because of the correspondence
\[
 V_1 = E_2 \ox K(k\cdot p)
 \word{and}
 V_2 = E_1
\]
through the isomorphism
$
 F_{d_1}^k \cong \sN_{\sg_H(k)}(2,1,\td_1,\td_2)
$
where
\[
 \td_1 = \deg(V_1) = \deg\big(E_2 \ox K(k\cdot p)\big) = d_2 + 2\sg_H(k)
 \word{and}
 \td_2 = \deg(V_2) = \deg(E_1) = d_1.
\]
 Similarly, the notation becomes
\[
 V_1 = E_1 \ox K(k\cdot p)
 \word{and}
 V_2 = E_2
\]
through the isomorphism
$
 F_{d_2}^k \cong \sN_{\sg_H(k)}(1,2,\td_1,\td_2)
$
for the dual cases, and so
\[
 \td_1 = \deg(V_1) = \deg(E_1 \ox K(k\cdot p)) = d_1 + \sg_H(k)
 \word{and}
 \td_2 = \deg(V_2) = \deg(E_2) = d_2.
\]

\begin{Th}\label{CohomologyNegativeFlipLocus}
 There is an isomorphism
 \[
  i_k^*: H^{j}(S_{-}^{k+1},\mathbb{Z}) 
 \xrightarrow{\quad \cong \quad} 
 H^{j}(S_{-}^{k},\mathbb{Z})
 \]
 for all 
 $
 j \leqslant \td_1 - d_M - \td_2 - 1 
 = d_2 - d_1 + 2\sg_H(k) - d_M
 $, 
 where $d_j = \deg(E_j)
 $,
 $\td_j = \deg(V_j)$, 
 $M\to X$ is a line bundle of degree
 $d_M=\deg(M)$, and 
 $\sg_H(k)= \deg(K(kp)) = 2g-2+k$.
\end{Th}

\begin{proof}
 Recall that, according to~\cite[Theorem~4.8.]{mov}, 
 $S_{-}^{k} = \mathbb{P}(\mathcal{V})$ is the 
 projectivization of a bundle 
 $\mathcal{V} \to \sN'_{\sg_c} \times \sN''_{\sg_c}$ 
 of rank $\mathrm{rk}(\mathcal{V}) = -\chi(T'',T')$, where 
 \[
  \sN'_{\sg_c} = \sN_{\sg_c}(1,1,\td_1-d_M,\td_2) 
 \cong 
 \mathcal{J}^{\td_2} \times \mathrm{Sym}^{\td_1-d_M-\td_2}(X)
 \]
 and
 \[
 \sN''_{\sg_c} = 
 \sN_{\sg_c}(1,0,d_M,0) 
 \cong \mathcal{J}^{d_M}(X)
 \]
 where any triple 
 $
 T = (V_1,V_2,\varphi) \in S_{-}^{k} \subseteq 
 \sN_{\sg_c^{-}(k)}(2,1,\td_1,\td_2)
 $ 
 is a non-trivial extension of a subtriple 
 $ T' \subseteq T$ of the form 
 $
 T' = (V'_1,V'_2,\varphi') 
 = (L,V_2,\varphi')
 $ 
 by a quotient triple of the form 
 $
 T'' = (V''_1,V''_2,\varphi'') = 
 (M,0,\varphi'')$, where $M
 $ 
 is a line bundle of degree $\deg(M) = d_M$ and 
 $L$ is a line bundle of degree $\deg(L) = d_L = \td_1 - d_M$. 
 
 \noindent Then, the embedding $i_k: T \rightarrow i_k(T)$ restricts to:
 \begin{align*}
  \begin{xy}
    (0,0)*+{\big([V_2'],\mathrm{div}(\varphi')\big)}="a0";
    (0,10)*+{\mathcal{J}^{\td_2} \times \mathrm{Sym}^{\td_1-d_M-\td_2}(X)}="a1";
    (0,40)*+{\sN'_{\sg_c}}="a2";
    (0,50)*+{(V_1',V_2',\varphi')}="a3";
    (70,0)*+{\big([V_2'],\mathrm{div}(\varphi'\ox s_p)\big)}="b0";
    (70,10)*+{\mathcal{J}^{\td_2} \times \mathrm{Sym}^{\td_1-d_M-\td_2+1}(X)}="b1";
    (70,40)*+{\sN'_{\sg_c+1}}="b2";
    (70,50)*+{(V_1'\ox L_p,V_2',\varphi'\ox s_p)}="b3";
    {\ar@{<->}^{\cong} "a1";"a2" **\dir{--}};
    {\ar@{<->}_{\cong} "b1";"b2" **\dir{--}};
    {\ar@{->}_{i_k} "a1";"b1" **\dir{--}};
    {\ar@{->}^{i_k} "a2";"b2" **\dir{--}};
    {\ar@{|->} "a0";"b0" **\dir{--}};
    {\ar@{|->} "a3";"b3" **\dir{--}};
  \end{xy}
 \end{align*}
 because $\sg_c(k+1) = \sg_c(k) + 1$, and $d_M(k+1) = d_M(k) + 1$, 
 and because, by the proof of the Roof Theorem \ref{RoofTheorem}, 
 $i_k$ restricts to the flip locus $S_{-}^{k}$.
 
 \noindent Recall that in our case $\sg_c = \sg_c(k) > \sg_m$. Then, for subtriples of the form $T' = (V_1',V_2',\varphi')$ we get that 
 $\varphi' \neq 0$ and so, they are entirely parametrized by $\big([V_2'], \mathrm{div}(\varphi)\big)$. That is why the map from $\sN'_{\sg_c}$ 
 to $\sJ^{\td_2} \times \mathrm{Sym}^{\td_1-d_M-\td_2}(X)$ is an isomorphism at the Jacobian.
 
 \noindent Similarly, $i_k$ restricts to:
 \begin{align*}
  \begin{xy}
    (0,0)*+{[V_1'']}="a0";
    (0,10)*+{\mathcal{J}^{d_M}}="a1";
    (0,40)*+{\sN''_{\sg_c}}="a2";
    (0,50)*+{(V_1'',0,0)}="a3";
    (50,0)*+{[V_1''\ox L_p]}="b0";
    (50,10)*+{\mathcal{J}^{d_M}}="b1";
    (50,40)*+{\sN''_{\sg_c+1}}="b2";
    (50,50)*+{(V_1''\ox L_p,0,0)}="b3";
    {\ar@{<->}^{\cong} "a1";"a2" **\dir{--}};
    {\ar@{<->}_{\cong} "b1";"b2" **\dir{--}};
    {\ar@{->}_{i_k} "a1";"b1" **\dir{--}};
    {\ar@{->}^{i_k} "a2";"b2" **\dir{--}};
    {\ar@{|->} "a0";"b0" **\dir{--}};
    {\ar@{|->} "a3";"b3" **\dir{--}};
  \end{xy}
 \end{align*}
 Here, the quotient triples of the form $T'' = (V_1'',0,0)$ are trivially parametrized by $[V_1'']$ and so, the map from $\sN''_{\sg_c}$ to 
 $\sJ^{d_M}$ is also an isomorphism at the Jacobian.
 
 \noindent Hence, by Corollary~\ref{macdonald-12.2}, 
 \[
  i_k^*: H^{j}(\sN'_{\sg_c +1},\mathbb{Z}) 
 \xrightarrow{\quad \cong \quad} 
 H^{j}(\sN'_{\sg_c},\mathbb{Z})\quad
 \forall j \leqslant \td_1 - d_M - \td_2 - 1,
 \] 
 and hence
 \[
  i_k^*: H^{j}(S_{-}^{k+1},\mathbb{Z}) 
 \xrightarrow{\quad \cong \quad} 
 H^{j}(S_{-}^{k},\mathbb{Z})\quad 
 \forall j \leqslant \td_1 - d_M - \td_2 - 1.
 \qed
 \]
 \hideqed
\end{proof}

\noindent Similarly, for the flip locus $S_{+}^{k} = S_{\sg_c^+(k)}$ we have:

\begin{Th}\label{CohomologyPositiveFlipLocus}
 There is an isomorphism
 \[
  i_{k}^{*}: H^{j}(S_{+}^{k+1},\mathbb{Z}) 
 \xrightarrow{\quad \cong \quad} 
 H^{j}(S_{+}^{k},\mathbb{Z})
 \]
 for all 
 $
 j \leqslant \td_1 - d_M - \td_2 - 1 
 = d_2 - d_1 + 2\sg_H(k) - d_M
 $, 
 where $d_j = \deg(E_j)
 $,
 $\td_j = \deg(\tilde{E}_j)$, 
 $M\to X$ is a line bundle of degree
 $d_M=\deg(M)$, and 
 $\sg_H(k)= \deg(K(kp)) = 2g-2+k$.
\end{Th}

\begin{proof}
 Quite similar argument to the one presented above, except 
 for the detail that this time is the other way around: 
 according also to~\cite[Theorem~4.8.]{mov}, 
 $
 S_{+}^{k} = \mathbb{P}(\mathcal{V})
 $ 
 is the projectivization of a bundle 
 $
 \mathcal{V} \to \sN'_c \times \sN''_c
 $ 
 of rank $\mathrm{rk}(\mathcal{V}) = -\chi(T'',T')$, 
 but this time
 $
 \sN'_c = \sN_c(1,0,d_M,0) 
 \cong \mathcal{J}^{d_M}(X)$,
 and 
 $
 \sN''_c = \sN_c(1,1,\td_1-d_M,\td_2) 
 \cong \mathcal{J}^{\td_2} \times \mathrm{Sym}^{\td_1-d_M-\td_2}(X)
 $
 where any triple 
 $
 T = (V_1,V_2,\varphi) \in S_{+}^{k} 
 \subseteq 
 \sN_{\sg_c^{+}(k)}(2,1,\td_1,\td_2)
 $ 
 is a non-trivial extension of a subtriple 
 $T' \subseteq T$ of the form $T' = (V'_1,V'_2,\varphi') = (M,0,\varphi')$ 
 by a quotient triple of the form 
 $T'' = (V''_1,V''_2,\varphi'') = (L,V_2,\varphi'')$, 
 where $M$ is a line bundle of degree $\deg(M) = d_M$ and 
 $L$ is a line bundle of degree $\deg(L) = d_L = \td_1 - d_M$.
\end{proof}

\begin{Th}\label{CohomologyNegative}
 There is an isomorphism
 \[
  i_k^*: H^{j}(\sN_{\sg^{-}_c(k+1)},\mathbb{Z}) 
 \xrightarrow{\quad \cong \quad} 
 H^{j}(\sN_{\sg^{-}_c(k)},\mathbb{Z})\quad 
 \forall j \leqslant 2\big(\td_1 - 2\td_2 - (2g - 2)\big) + 1.
 \]
\end{Th}

\noindent Since the behavior of $\sN_{\sg^{-}_c}$, 
where $\sg^{-}_c = \sg_c - \varepsilon$, is 
the same that the one of 
$\sN_{\sg^{+}_m}$, where 
$\sg^{+}_m = \sg_m + \varepsilon$, is enough 
to prove the following lemma:

\begin{Lem}
 The relative cohomology groups
 \[
  H^{j}(\sN_{\sg^{+}_m(k+1)},\sN_{\sg^{+}_m(k)};\mathbb{Z}) = 0
 \]
 are trivial for all 
 $
 j \leqslant 2\big(\td_1 - 2\td_2 - (2g - 2)\big).
 $
\end{Lem}

\begin{proof}
 Note that $\sN_{\sg^{-}_m(k)} = \emptyset$, 
 hence $\sN_{\sg^{+}_m(k)} = S_{+}^{k}$, and 
 according to~\cite[Theorem~4.10.]{mov}, any triple 
 $
 T = (V_1,V_2,\varphi) \in S_{+}^{k} = 
 \sN_{\sg_m^{+}(k)}(2,1,\td_1,\td_2)
 $ 
 is a non-trivial extension of a subtriple $T' \subseteq T$ 
 of the form $T' = (V'_1,V'_2,\varphi') = (V_1,0,0)$ 
 by a quotient triple of the form 
 $T'' = (V''_1,V''_2,\varphi'') = (0,V_2,0)$. Hence, 
 there is a map
 \[
  \pi: \sN_{\sg^{+}_m} \to 
 \sN(2,\td_1) \times \mathcal{J}^{\td_2}(X)
 \]
 \[
  (V_1,V_2,\varphi) \mapsto ([V_1],[V_2])
 \]
 where the inverse image 
 $
 \pi^{-1}\big( \sN(2,\td_1) \times 
 \mathcal{J}^{\td_2}(X) \big)= \mathbb{P}^{N}
 $
 has rank
 $
 N = -\chi(T'',T') = \td_1 - 2\td_2 - (2g-2)
 $, and the proof follows.
\end{proof}

\begin{Th}\label{CohomologyBlowUp}
 There is an isomorphism
 \[
  \tilde{i_k^*}\: 
 H^{j}(\tilde{\sN}_{\sg_c(k+1)},\mathbb{Z}) 
 \xrightarrow{\quad \cong \quad} 
 H^{j}(\tilde{\sN}_{\sg_c(k)},\mathbb{Z})\quad 
 \forall j \leqslant n(k)
 \]
 at the blow-up level, where 
 $
 n(k) = 
 \min(\td_1 - d_M - \td_2 - 1,\quad 2\big(\td_1 - 2\td_2 - (2g - 2)\big) + 1)
 $.
\end{Th}

\begin{proof}
 By the Roof Theorem~\ref{RoofTheorem}, $i_k$ 
 lifts to the blow-up level. We will denote 
 $
 \sN^{k}_{-} = 
 \sN_{\sg^{-}_c(k)}(2,1,\td_1,\td_2)
 $ and 
 $\tilde{\sN}^{k} = \tilde{\sN}_{\sg_c(k)}$ its blow-up along 
 the flip locus $S_{-}^{k} = S_{\sg_c^-(k)}$. Recall that, from the construction of the 
 blow-up, there is a map $\pi_{-}: \tilde{\sN}^{k} \to \sN^{k}_{-}$ 
 such that
 \[
  0\to 
 \pi_{-}^{*}\big( H^{j}(\sN^{k}_{-}) \big)\to 
 H^{j}(\tilde{\sN}^{k})\to H^{j}(\sE^{k}) / \pi_{-}^{*}\big( H^{j}(\mathcal{S}^{k}_{-}) \big)
 \to 0
 \]
 splits where $\sE^{k} = \pi_{-}^{-1}(S_{-}^{k})$ is the so-called exceptional divisor. Hence, the following diagram
 \begin{align}
    \begin{xy}
  (-20,0)*+{0}="a1";
  (10,0)*+{\pi_{-}^{*}\big( H^{j}(\sN^{k}_{-}) \big)}="b1";
  (50,0)*+{H^{j}(\tilde{\sN}^{k})}="c1";
  (90,0)*+{H^{j}(\sE^{k}) / \pi_{-}^{*}\big( H^{j}(\mathcal{S}^{k}_{-}) \big)}="d1";
  (120,0)*+{0}="e1";
  (-20,-20)*+{0}="a2";
  (10,-20)*+{\pi_{-}^{*}\big( H^{j}(\sN^{k+1}_{-}) \big)}="b2";
  (50,-20)*+{H^{j}(\tilde{\sN}^{k+1})}="c2";  
  (90,-20)*+{H^{j}(\sE^{k+1}) / \pi_{-}^{*}\big( H^{j}(\mathcal{S}^{k+1}_{-}) \big)}="d2";
  (120,-20)*+{0}="e2";
  {\ar@{->} "a1";"b1"};
  {\ar@{->} "b1";"c1"};
  {\ar@{->} "c1";"d1"};
  {\ar@{->} "d1";"e1"};
  {\ar@{->} "a2";"b2"};
  {\ar@{->} "b2";"c2"};
  {\ar@{->} "c2";"d2"};
  {\ar@{->} "d2";"e2"};
  {\ar@{->}^{\cong} "b2";"b1"};
  {\ar@{-->}^{\tilde{{i_k^*}}} "c2";"c1"};
  {\ar@{->}_{\cong} "d2";"d1"};
    \end{xy}
 \end{align}
 commutes for all $j \leqslant n(k)$, and the theorem follows.
\end{proof}

\begin{Cor}\label{CohomologyPositive}
 There is an isomorphism
 \[
  i_k^*: H^{j}(\sN_{\sg^{+}_c(k+1)},\mathbb{Z}) 
 \xrightarrow{\quad \cong \quad} 
 H^{j}(\sN_{\sg^{+}_c(k)},\mathbb{Z})\quad 
 \forall j \leqslant n(k)
 \]
 where 
 $
 n(k) = 
 \min(\td_1 - d_M - \td_2 - 1,\quad 2\big(\td_1 - 2\td_2 - (2g - 2)\big) + 1)
 $ as before.
\end{Cor}

\begin{proof}
 Recall that $\tilde{\sN}^{k} = \tilde{\sN}_{\sg_c(k)}$ is also the blow-up 
 of $\sN^{k}_{+} = \sN_{\sg^{+}_c(k)}(2,1,\td_1,\td_2)$ along 
 the flip locus $S_{+}^{k} = S_{\sg_c^+(k)}$, so there is a map $\pi_{+}: \tilde{\sN}^{k} \to \sN^{k}_{+}$ 
 such that
 \[
 0\to \pi_{+}^{*}\big( H^{j}(\sN^{k}_{+}) \big)\to 
 H^{j}(\tilde{\sN}^{k})\to 
 H^{j}(\sE^{k}) / \pi_{+}^{*}\big( H^{j}(\mathcal{S}^{k}_{+}) \big)\to 0 
 \]
 splits:
 \[
 H^{j}(\tilde{\sN}^{k}) = 
 \pi_{+}^{*}\big( H^{j}(\sN^{k}_{+}) \big) 
 \oplus 
 H^{j}(\sE^{k}) / \pi_{+}^{*}\big( H^{j}(\mathcal{S}^{k}_{+}) \big), 
 \]
 and by Theorem~\ref{CohomologyPositiveFlipLocus} and 
 Theorem~\ref{CohomologyBlowUp}, the result follows.
\end{proof}

\begin{Cor}\label{CohomologyCriticalTriples}
 There is an isomorphism
 \[
 i_k^*\: 
 H^{j}(\sN_{\sg_c(k+1)},\mathbb{Z}) \xrightarrow{\quad \cong \quad} 
 H^{j}(\sN_{\sg_c(k)},\mathbb{Z})\quad \forall j \leqslant n(k).
 \eqno \QEDA 
 \]
\end{Cor}

\subsection[Cohomology of (1,2)-VHS]{Cohomology of the $(1,2)$-VHS}
\label{ssec:3.3}

So far, we stabilize the cohomology of $\sN_{\sg_c(k)}$ for any critical $\sg_c(k)$.
Here and after, $\sg_L$ respresents the largest critical value in the open interval
$]\sg_m,\ \sg_M[$, and $\sN_{\sg_L^+}$ (respectively $\sN_{\sg_L^+}^s$) denotes the moduli space 
of $\sg_L$-polystable (respectively $\sg_L$-stable) triples for values
$\sg_L < \sg < \sg_M$. The space $\sN_{\sg_L^+}$ is so-called the `large $\sg$' moduli space 
(see \cite{bgg1}).
The following results will 
allow us to generalize the stabilization for all 
$
\displaystyle
\sg \in \left]\sg_m(k),\ \sg_M(k)\right[
$:

\begin{Th}[{\cite[Th.~7.7.]{bgg1}}]\label{bgg1_7.7}
Assume that $r_1 > r_2$ and $\displaystyle \frac{d_1}{r_1} > \frac{d_2}{r_2}$. 
Then the moduli space $\sN_{\sg_L^+}^s = \sN_{\sg_L^+}^s(r_1,r_2,d_1,d_2)$ is smooth of dimension
\[
(g - 1)(r_1^2 + r_2^2 - r_1 r_2)- r_1 d_2 + r_2 d_1 +1,
\]
and is birationally equivalent to a $\bP^{\tilde{n}}$-fibration over 
$
\sN^s(r_1 - r_2, d_1 - d_2)\times \sN^s(r_2, d_2),
$
where $\sN^s(r,d)$ is the moduli space of stable bundles of degree $r$ and degree
$d$, and 
\[
\tilde{n} = r_2 d_1 - r_1 d_2 +r_1(r_1 - r_2)(g - 1) - 1.
\]
In particular, $\sN_{\sg_L^+}^s(r_1,r_2,d_1,d_2)$ is non-empty and irreducible.

\noindent If $\GCD(r_1 - r_2, d_1 - d_2) = 1$ and $\GCD(r_2, d_2) = 1$, the birational 
equivalence is an isomorphism.

\noindent Moreover, in all cases, $\sN_{\sg_L^+} = \sN_{\sg_L^+}(r_1,r_2,d_1,d_2)$ is irreducible and hence,
birationally equivalent to $\sN_{\sg_L^+}^s$.
\QEDA
\end{Th}

\begin{Th}[{\cite[Th.~7.9.]{bgg1}}]\label{bgg1_7.9}
 Let $\sg$ be any value in the range $\sg_m < 2g - 2 \leq \sg < \sg_M$, then 
 $\sN^s_{\sg}$ is birationally equivalent to $\sN_{\sg_L^+}^s$. In particular it is non-empty 
 and irreducible. 
 \QEDA
\end{Th}

\begin{Cor}[{\cite[Cor.~7.10.]{bgg1}}]\label{bgg1_7.10}
 Let $(\rr,\dd) = (r_1,r_2,d_1,d_2)$ be such that 
 $$
 \GCD(r_2, r_1 + r_2, d_1 + d_2) = 1.
 $$ 
 If $\sg$ is a generic value satisfying $\sg_m < 2g - 2 \leq \sg < \sg_M$, 
 then $\sN_{\sg}$ is birationally equivalent to $\sN_{\sg_L^+}$, and in particular it is 
 irreducible. 
\end{Cor}

\begin{proof}
$\sN_{\sg} = \sN^s_{\sg}$ if 
$
 \GCD(r_2, r_1 + r_2, d_1 + d_2) = 1
$
and $\sg$ is generic. In particular, we have
$\sN_{\sg_L^+} = \sN_{\sg_L^+}^s$, and the result follows from the last theorem. 
The reader may see the full details in \cite{bgg1}.
\end{proof}

\begin{Th}\label{z-r_01}
There is an isomorphism
\[
 i_k^*\: 
 H^{j}(\sN_{\sg_H}^{k+1},\mathbb{Z}) 
 \xrightarrow{\quad \cong \quad} 
 H^{j}(\sN_{\sg_H}^{k},\mathbb{Z})\quad 
 \forall j \leqslant \sg_H(k) - 2(\mu_1 - \mu) - 1
\]
where
$
\sN_{\sg_H}^{k} = 
\sN_{\sg_{H}(k)}(2,1,\td_1,\td_2)
$,\quad
$
\sg_{H} = 
\sg_{H}(k) 
= 2g - 2 + k
$,
and $\mu_1 = \mu(E_1) > \mu(E) = \mu$.
\end{Th}

\begin{proof}
 In this case $\GCD(1,3,\td_1 + \td_2) = 1$ trivially, and 
 $\sg_H = \sg_H(k)$ is a $\sg$-critical value that satisfies
 \[
 \sg_m < 2g - 2 \leq \sg_H(k) < \sg_M.
 \]
Therefore, by the description of Mu\~noz~et al.~\cite{mov} of the critical values (\cite{mov} Lemma~5.2. and Lemma~5.3.), 
the line bundle $M\to X$ satisfies in this case, the following:
\[
 \sg_m <
 \sg_H(k) =
 3d_M - \td_1 - \td_2
\]
equivalently
\[
 d_M = \sg_H(k) + \mu
\]
and hence
\[
 \td_1 - d_M - \td_2 - 1 =
 \sg_H(k) - 2(\mu_1 - \mu) - 1.
\]
In such a case
\[
 \td_1 - 2\td_2 - (2g-2) =
 \sg_H(k) - 2(\mu_1 - \mu) + k
 \geq 
 \sg_H(k) - 2(\mu_1 - \mu) - 1
\]
and then
\[
 2(\td_1 - 2\td_2 - (2g-2))
 \geq
 \td_1 - d_M - \td_2 - 1.
\]
Therefore, in this case
\[
 n(k) = 
 \td_1 - d_M - \td_2 - 1
 =
 \sg_H(k) - 2(\mu_1 - \mu) - 1.
\]

\noindent Finally, by Theorem~\ref{bgg1_7.9} and by Corollary~\ref{bgg1_7.10}, 
the space
 $
\sN_{\sg_H}^{k} = 
\sN_{\sg_{H}(k)}(2,1,\td_1,\td_2)
 $
 is birationally equivalent to 
 $
\sN_{\sg_L^{+}(k)} = 
\sN_{\sg_L^{+}(k)}(2,1,\td_1,\td_2)
 $,
 which is equal to the moduli space
 $
\sN_{\sg_L^{+}(k)}^s = 
\sN_{\sg_L^{+}(k)}^s(2,1,\td_1,\td_2)
 $
 of holomorphic stable triples
 also by Theorem~\ref{bgg1_7.9}, where $\sg_L^{+}(k)$ is the maximal critical 
 value, depending on $k$ in this case.
 The isomorphism then follows by Corollary~\ref{CohomologyCriticalTriples}.
\end{proof}

\begin{Cor}\label{(1,2)-VHS--Cohomology}
 There is an isomorphism 
 \[
 H^{j}(F_{d_1}^{k+1},\mathbb{Z}) 
 \xrightarrow{\quad \cong \quad} 
 H^{j}(F_{d_1}^{k},\mathbb{Z})  
 \]
 for all
 $
 j \leqslant \sg_H(k) - 2(\mu_1 - \mu) - 1
 $
 induced by the embedding~\ref{eq:(1,2)-VHS-embedding}.
 \QEDA
\end{Cor}

\subsection[Cohomology of (2,1)-VHS]{Cohomology of the $(2,1)$-VHS}
\label{ssec:3.4}

Because of the duality 
\[
 \sN_{\sg}(1,2,\td_1,\td_2) 
 \cong 
 \sN_{\sg}(2,1,-\td_2,-\td_1)
\]
from Proposition~\ref{dualtriples},
we get
\begin{Th}
 There is an isomorphism
 \[
 i_k^*\: 
 H^{j}(\sN_{\sg_c(k+1)},\mathbb{Z}) \xrightarrow{\quad \cong \quad} 
 H^{j}(\sN_{\sg_c(k)},\mathbb{Z})\quad \forall j \leqslant m(k) 
 \]
 where 
 $
 m(k) = 
 \min(-\td_1 - d_M + \td_2 - 1,\quad 2\big(-\td_1 + 2\td_2 - (2g - 2)\big) + 1)
 $.
\end{Th}

\begin{proof}
The result follows as the analogous to Corollary~\ref{CohomologyCriticalTriples} varying $\td_1$ and $\td_2$ 
according to the duality from Proposition~\ref{dualtriples}. 
\end{proof}

\begin{Th}\label{z-r_02}
For $k$ large enough, there is an isomorphism
\[
 i_k^*\: 
 H^{j}(\sN_{\sg_H}^{k+1},\mathbb{Z}) 
 \xrightarrow{\quad \cong \quad} 
 H^{j}(\sN_{\sg_H}^{k},\mathbb{Z})\quad 
 \forall j \leqslant \sg_H(k) - 4(\mu_2 - \mu) - 1
\]
where
$
\sN_{\sg_H}^{k} = 
\sN_{\sg_{H}(k)}(1,2,\td_1,\td_2)
$,\quad
$
\sg_{H} = 
\sg_{H}(k) 
= 2g - 2 + k
$,
and $\mu_2 = \mu(E_2) > \mu(E) = \mu$.
\end{Th}

\begin{proof}
 In this case, by the duality from Proposition~\ref{dualtriples}, and by the description of the $\sg_c$ critical values 
 (Mu\~noz~et al.~\cite{mov} Lemma~5.2. and Lemma~5.3.), the line bundle $M\to X$ satisfies in this case, the following:
\[
 \sg_m <
 \sg_H(k) =
 3d_M + \td_1 + \td_2
\]
equivalently
\[
 d_M = -\mu > -\mu_2
\]
and hence
\[
 -\td_1 - d_M + \td_2 - 1 =
 \sg_H(k) - 4(\mu_2 - \mu) - 1,
\]
where, once again, $\sg_H = \sg_H(k)$ is a $\sg$-critical value satisfying
 \[
 \sg_m < 2g - 2 \leq \sg_H(k) < \sg_M.
 \]

\noindent In such a case
\[
 -\td_1 + 2\td_2 - (2g-2) =
 \sg_H(k) - 6(\mu_2 - \mu) + k \geq 
\]
\[
 \sg_H(k) - 4(\mu_2 - \mu) 
 \geq 
 \sg_H(k) - 4(\mu_2 - \mu) - 1
\]
if $k > 2(\mu_2 - \mu) > 0$ is large enough. Then
\[
 2(-\td_1 + 2\td_2 - (2g-2))
 \geq
 -\td_1 - d_M + \td_2 - 1.
\]
Therefore, in this case
\[
 m(k) = 
 -\td_1 - d_M + \td_2 - 1
 =
 \sg_H(k) - 4(\mu_2 - \mu) - 1.
\]

\noindent Hence, the result follows as the dual analogous to Theorem~\ref{z-r_01}.
\end{proof}

\begin{Cor}\label{(2,1)-VHS--Cohomology}
 For $k$ large enough, there is an isomorphism 
 \[
 H^{j}(F_{d_2}^{k+1},\mathbb{Z}) 
 \xrightarrow{\quad \cong \quad} 
 H^{j}(F_{d_2}^{k},\mathbb{Z})  
 \]
 for all
 $
 j \leqslant \sg_H(k) - 4(\mu_2 - \mu) - 1
 $
 induced by the embedding~\ref{eq:(2,1)-VHS-embedding}.
 \QEDA
\end{Cor}

\subsection[Cohomology of (1,1,1)-VHS]{Cohomology of the $(1,1,1)$-VHS}
\label{ssec:3.5}

\begin{Th}\label{(1,1,1)-VHS-pullback}
 The pull-back
 \[
  i_{k}^{*}\:
  H^{*}(F_{m_1 m_2}^{k+1},\bZ)
  \to
  H^{*}(F_{m_1 m_2}^{k},\bZ)
 \]
induced by the embedding $i_{k}\: F_{m_1 m_2}^{k}\to F_{m_1 m_2}^{k+1}$, is surjective.
\end{Th}

\begin{proof}
 This is a direct consequence of Theorem~\ref{(1,1,1)-VHS--SymP-Iso}, Theorem~\ref{symmetric-pullback} and Corollary~\ref{macdonald-12.2}.
\end{proof}

\begin{Cor}\label{(1,1,1)-cohomology-iso}
 There is an isomorphism
 \[
  H^{j}(F_{m_1 m_2}^{\infty},\bZ)
  \xrightarrow{\quad \cong \quad} 
  H^{j}(F_{m_1 m_2}^{k},\bZ)
 \]
 for all $j \leq \min \big(\bar{m}_1 + k, \bar{m}_2 + k\big) - 1$.
\end{Cor}

\begin{proof}
 It follows directly from Theorem~\ref{(1,1,1)-VHS--SymP-Iso}, Corollary~\ref{macdonald-12.2} 
 and Corollary~\ref{symmetric-cohomology-iso}.
\end{proof}

\section*{Acknowledgement}

I would like to thank Peter B. Gothen for introducing me to the
beautiful subject of Higgs bundles. I thank Vicente Mu\~noz and 
Andr\'e Gamma Oliveira for enlightening discussions about the 
moduli space of triples; I thank Steven Bradlow too, for the time 
and discussions about stable pairs and triples.
I am grateful to Joseph C. V\'arilly for helpful discussions.

\noindent Financial support from Funda\c{c}\~ao para a~Ci\^encia e~a~Tecnologia (FCT), 
and from Vicerrector\'ia de Investigaci\'on
de la Universidad de Costa Rica, is acknowledged.

\renewcommand*{\refname}{}


\section*{References}
\addcontentsline{toc}{section}{References}

\begin{thebibliography}{28}

  

\bibitem{acgh}
E.~Arbarello, M.~Cornalba, P.~Griffiths, J.D.~Harris,
\emph{Geometry of Algebraic Curves},
Vol.~I, Springer-Verlag, New~York, 1985.

  


\bibitem{atbo} 
M. F. Atiyah and R. Bott,
``Yang--Mills equations over Riemann surfaces'',
\emph{Phil. Trans. Roy. Soc. London A} \textbf{308} (1982), 523--615.


\bibitem{ben} 
S.~Bento, 
\emph{``Topologia do Espa\c{c}o Moduli de Fibrados de Higgs Torcidos''},
Tese de Doutoramento, Universidade do Porto, Porto, Portugal, 2010.


\bibitem{brgp} 
S. B. Bradlow and O. Garc\'ia-Prada,
``Stable triples, equivariant bundles and dimensional reduction'',
\emph{Math. Ann.} \textbf{304} (1996), 225--252.


\bibitem{bgg1} 
S. B. Bradlow, O. Garc\'ia-Prada and P. B. Gothen,
``Moduli spaces of holomorphic triples over compact Riemann surfaces'',
\emph{Math. Ann.} \textbf{328} (2004), 299--351.


\bibitem{bgg2} 
S. B. Bradlow, O. Garc\'ia-Prada and P. B. Gothen,
``Homotopy groups of moduli spaces of representations'',
\emph{Topology} \textbf{47} (2008), 203--224.


\bibitem{decataldo-mark-hausel-migliorini:2012}
M.~A. de~Cataldo, T.~Hausel, and L.~Migliorini,
{``Topology of {H}itchin systems and {H}odge theory of character varieties: the case {$A_1$}''}, 
\emph{Ann. of Math.} (2) \textbf{175} (2012), no.~3, 1329--1407. 
  

\bibitem{dera}
U.V.~Desale, S.~Ramanan,
{``Poincar\'e polynomials of the variety of stable bundles''},
\emph{Math.Ann.} \textbf{216} (1975), no.~3, 233--244. 


\bibitem{eaki}
R.~Earl, F.~Kirwan,
{``The Hodge numbers of the moduli spaces of vector bundles over a Riemann surface''},
\emph{Q.J.~Math.} {\bf 51} (2000), no.~4, 465--483. 

  
\bibitem{fra} 
T. Frankel,
``Fixed points and torsion on K\"ahler manifolds'',
\emph{Ann. Math.} \textbf{70} (1959), 1--8.


\bibitem{ggm}
O.~Garc\'ia-Prada, P.B.~Gothen, V.~Mu\~noz,
{``Parabolic Higgs bundles''}
\emph{AMS Memoirs}

  
\bibitem{got} 
P. B. Gothen,
``The Betti numbers of the moduli space of stable rank~3 Higgs bundles
on a Riemann surface'', 
\emph{Int. J. Math.} \textbf{5} (1994), 861--875.


\bibitem{grha}
P. Griffiths, and J. Harris,
\emph{Principles of Algebraic Geometry},
Wiley, New York, 1978.


\bibitem{hau} 
T.~Hausel,
\emph{``Geometry of Higgs bundles''}, 
Ph.D.~Thesis, Cambridge, 1998.


\bibitem{hausel:2013}
T.~Hausel, 
\emph{Global topology of the {H}itchin system}, 
Handbook of  moduli. {V}ol. {II}, Adv. Lect. Math. (ALM), 
vol.~25, Int. Press, Somerville,  MA, 2013, pp.~29--69. 

  
\bibitem{hausel-letellier-rodriguez-villegas:2011}
T.~Hausel, E.~Letellier, and F.~Rodr\'iguez-Villegas,  
{``Arithmetic harmonic analysis on character and quiver varieties''}, 
\emph{Duke Math. J.} \textbf{160} (2011), no.~2, 323--400. 

  
\bibitem{hausel-rodriguez-villegas:2008}
T.~Hausel and F.~Rodr\'iguez-Villegas, 
{``Mixed {H}odge polynomials of character varieties''}, 
\emph{Invent. Math.} \textbf{174} (2008), no.~3, 555--624, 
with an appendix by Nicholas M. Katz. 


\bibitem{hath1} 
T. Hausel and M. Thaddeus,
``Generators for the cohomology ring of the moduli space of rank~2
Higgs bundles'',
\emph{Proc. London Math. Soc.} \textbf{88} (2004), 632--658.


\bibitem{hath2} 
T. Hausel and M. Thaddeus,
``Relations in the cohomology ring of the moduli space of rank~2
Higgs bundles'',
\emph{J. Amer. Math. Soc.} \textbf{16} (2003), 303--329. 


\bibitem{hit2} 
N. J. Hitchin, 
``The self-duality equations on a Riemann surface'',
\emph{Proc. London Math. Soc.} \textbf{55} (1987), 59--126.


\bibitem{mac} 
I. G. Macdonald, 
``Symmetric products of an algebraic curve'',
\emph{Topology} \textbf{1} (1962), 319--343.


\bibitem{mos} 
V. Mu\~noz, A. Oliveira and J. S\'anchez,
``Motives and the Hodge Conjecture for the Moduli Spaces of Pairs'', 
\emph{Asian Journal of Mathematics}, Vol. \textbf{19} (2015), 281--306. 


\bibitem{mov} 
V. Mu\~noz, D. Ortega and M. J. V\'azquez-Gallo,
``Hodge polynomials of the moduli spaces of pairs'',
\emph{Int. J. Math.} \textbf{18} (2007), 695--721.


\bibitem{nit} 
N. Nitsure, 
``Moduli space of semistable pairs on a curve'',
\emph{Proc. London Math. Soc.} \textbf{62} (1991), 275--300.


\bibitem{schmitt}
A.~Schmitt,
{``A universal construction for the moduli spaces of decorated vector bundles'',} {\em Transform. Groups}
{\bf 9} (2004), 162--209.


\bibitem{sim1} 
C. T. Simpson, 
``Constructing variations of Hodge structures using Yang--Mills theory
and applications to uniformization'',
\emph{J. Amer. Math. Soc.} \textbf{1} (1988), 867--918.


\bibitem{sim2} 
C. T. Simpson, 
``Higgs bundles and local systems'',
\emph{Publ. Math. IH\'ES} \textbf{75} (1992), 5--95.


\bibitem{z-r0}
R.A. Z\'u\~niga-Rojas, 
``Stabilization of the Homotopy Groups of The Moduli Space of k-Higgs Bundles'', 
\emph{Revista Colombiana de Matemáticas}~{\bf 52} (2018)~1, 9--31.

\bibitem{z-r}
R. A. Z\'u\~niga-Rojas,
\emph{``Homotopy groups of the moduli space of Higgs bundles''},
Ph.D.~Thesis, Porto, 2015.

\end{thebibliography}
\end{document}